\documentclass{article}
\usepackage{snapshot}
\usepackage{fullpage}

\usepackage{amsfonts}
\usepackage{amsthm}
\newcommand{\vol}{\textrm{vol}}
\newcommand{\E}{\textrm{E}}

\newcommand{\Var}{\textrm{Var}}
\newcommand{\area}{\textrm{area}}
\newtheorem{lemma}{Lemma}
\newtheorem{definition}{Definition}
\newtheorem{theorem}{Theorem}

\title{The Small-Community Phenomenon in
Networks}
\author{Angsheng Li\footnote{State Key Laboratory of
Computer Science, Institute of Software, Chinese Academy of Sciences, email:  angsheng@ios.ac.cn} \\Chinese Academy of Sciences
\and
Pan Peng\footnote{State Key Laboratory of
Computer Science, Institute of Software, Chinese Academy of Sciences, email:  pengpan@ios.ac.cn} \\Chinese Academy of Sciences
}

\begin{document}

\maketitle



\begin{abstract}We investigate several geometric models of network which
simultaneously have some nice global properties, that the small
diameter property, \textit{the small-community phenomenon}, which is
defined to capture the common experience that (almost) every one in
our society belongs to some meaningful \textit{small} communities by
the authors (2011), and that under certain conditions
on the parameters, the power law degree distribution, which
significantly strengths the results given by van den Esker (2008),
and Jordan (2010). The results above, together with our previous progress in Li and
Peng (2011), build a mathematical foundation for the study of
communities and the small-community phenomenon in various networks.

In the proof of the power law degree distribution, we develop the
method of {\it alternating concentration analysis} to build
concentration inequality by alternatively and iteratively applying
both the sub- and super-martingale inequalities, which seems
powerful, and which may have more potential applications.

\end{abstract}


\section{Introduction}
With the availability of massive datasets of many real world
networks, we are able to observe and study the underlying dynamic
mechanisms and many interesting phenomena in large-scale networks in
a quantitative way. Some properties such as sparse, high-clustering,
hierarchical structure, the power law degree distribution and small
diameter appear in a wide range of networks, ranging from Internet
graphs, collaboration graphs to PPI (Protein-Protein Interaction)
networks. Modeling these interesting properties and phenomena not
only provides us a good way to better understand how these networks
evolve and why these global phenomena occur through local growing
rules, but also gives us insights on the development of new
technologies or even cancer drugs.

A typical network always simultaneously exhibits several properties.
For example, in a Web graph, the nodes are web-pages and directed
edges are hyperlinks between the pages, the number of nodes with
in-degree $k$ is proportional to $k^{-\beta}$ for some constant
$\beta$, i.e., the in-degree sequence obeys the power law degree
distribution (\cite{AJB99:diameterWWW,KKRRT99:web}). It has also
been observed that the Web graphs have a small average
distance~\cite{AJB99:diameterWWW,BKMRRSTW00:structure}. In this
paper, when it is not confused, ``small'' means that the quantity is
a polylogarithmic function of the number of graph nodes.
Furthermore, the most community-like subgraphs in the large Web
graphs turn out to have size about $100$, which seems to be a
general property in many large real
networks~\cite{LLDM08:community,LLM10:empirical}. The above
mentioned three properties are by no means particular in the
technological networks, and they are shared also by a wide range of
social networks, such as the friendship network of LiveJournal.

The first two properties, i.e., the power law degree distribution
and the small diameter property, have been explored extensively in
the past decades. However, to our knowledge, the third property that
good communities in large-scale networks have small sizes is still
widely open due to the reason that there were no mathematical
definition of good communities in network, which motivates us to
mathematically study the common experiences, or observations, or
small experiments of the small-community phenomenon in networks.

 The authors proposed a mathematical definition for
communities based on the concept of conductance, defined the {\it
small-community phenomenon} in networks, and conjectured that small
communities are ubiquitous in various networks (referred to our earlier work~\cite{LP11:smallcommunity}). Intuitively, a given network is said to have the small-community phenomenon if almost every node in the network is
contained in some good community of small size (referred to
Section~\ref{sec:def} for the formal definition). We found
theoretical evidence for our conjecture that some classical network
models (e.g., Kleinberg's small world model~\cite{Kle00:smallworld}
and Ravasz-Barab\'{a}si Hierarchical model~\cite{RB03:hierarchy}) do
have the small-community phenomenon, though models without this
phenomenon exist.

There are also other reasons for us to make such a conjecture.
Firstly, we all have the common experience that everyone in the
society belong to some small meaningful groups which may correspond
to classmates, friends or relatives. Secondly, the existing
empirical studies provide us evidence that large communities are
rare in large networks and good communities are of small size.
Besides the direct evidence given
by~\cite{LLDM08:community,LLM10:empirical,GR09:small}, there are
also some implicit evidence. For example, \cite{Lan05:spectral} has
shown that spectral graph partitioning fails to generate highly
unbalanced cuts over many large scale social networks and
\cite{KB10:geocommunity} pointed out that this failure may be caused
by the abundance of small dense communities. In summary, we have
reasons to conjecture that the small communities are ubiquitous, at
least in many large social networks, which arises a number of new
problems in both theory and applications of the small-community
phenomenon in networks.

We are interested in \textit{evolving} models that simultaneously
have these ``good'' properties: the power law degree distribution,
the small diameter and the small-community phenomenon, which are
shared in typical Web graphs and large social networks. Models with
one or two properties are easily constructed in some natural way. In
particular, the power law degree distribution arises from the
preferential attachment scheme; the small diameter originates from a
broad class of graph processes~\cite{Bol01:randomgraphs}; the small
community may be caused by the notion of {\it homophily} that,
similar or close individuals have great tendency to associate with
each other, which is a common reason for two people establish a
relationship with each other in our society.

However, when trying to define a model that unify all the three
properties, we usually come across conflicts that are hard to
reconcile. Not strictly speaking, the first two properties usually
result from some expander like graphs while the small community
corresponds to highly structured graphs which seem anti-expander
like graphs to some extent~\cite{LP11:smallcommunity}. Still,
Ravasz-Barab\'{a}si Hierarchical model~\cite{RB03:hierarchy}
satisfies all these requirements as shown
in~\cite{LP11:smallcommunity}. However, the Ravasz-Barab\'{a}si
Hierarchical model has a very unnatural growing rule, which can only
capture very special networks.

Another good candidate may be the Geometric Preferential Attachment
(henceforth GPA) model introduced by Flaxman, Frieze and
Verta~\cite{FFV07:gpamodel1,FFV07:gpamodel2}, whose
motivation was to model networks with power law degree distribution
and small expansion. This model is defined on a unit-area spherical
surface $S$, on which distance can be naturally introduced. The
authors of~\cite{FFV07:gpamodel1} combined the rich-get-richer
effect and the concept of homophily in a simple way that every new
comer chooses neighbors only from those exiting vertices that are
not far from them using the preferential attachment scheme, and
proved that the power law distribution occurs under some conditions
of the parameters in the model. In~\cite{LP11:smallcommunity}, we
have shown that good communities exist for every node in a model
under these conditions. However, the found communities are of
relatively large size and the diameter is also large.

In the present paper, we will first study a base model that is a GPA
model with {\it additive fitness}. We generalize the result
of~\cite{Jor10:gpa} and show that under some appropriate conditions,
the base model have both the power law degree distribution and the
small-community phenomenon. However, in this situation the diameter
of the model is large. To resolve this problem, we try to
incorporate a simple growth rule into our base model that leads to
small diameter and does not change too much the degree sequence. The
rule we try to use is the uniform recursive tree, i.e., each time a
new vertex chooses a neighbor uniformly at random from exiting
vertices. It has been well known that such a simple process results
in a graph of diameter and maximum degree of order $\Theta(\ln n)$,
where $n$ is the number of generated
vertices~\cite{SM95:recursivetree}. We give two alternate ways to
incorporate this rule. Though the resulted two models are similar,
their structures are different. The first one is a {\it hybrid
model}, which can be regarded as a composition of two independent
parts: {\it a local graph}, which has the power law degree
distribution, and {\it a global graph}, which may connect vertices
that are far away. The hybrid model as a whole has the small
diameter and the small community structure. The second is a {\it
self-loop model}, in which we treat the additive fitness in our base
model as the number of self-loops attached with the new vertex. This
gives a new interpretation for the use of fitness in the
preferential attachment schemes. With some further operation, the
self-loop model is shown to have all the three good properties.

The methodology we use to show the power law degree
distribution may be of independent interest. The proof technique is
inspired by the work~\cite{Jor10:gpa}, who investigated the
\textit{asymptotic behavior} of the degree sequence of the base
model (see Section~\ref{sec:def}). In our proof of the concentration
inequalities, there are subtle restrictions on parameters for which
deeper mathematics is needed. Rather than using the coupling
techniques as that in~\cite{FFV07:gpamodel1,Esk08:geometric}, we
recursively utilize the submartingale and supermartingale
concentration inequalities~\cite{CL06:complex} to give a better bound
at each step, which will result in a sharp bound of the desired
quantity.

\textbf{Further related works} Avin studied a random
distance graph that incorporates both the Erd{\"o}s-R\'{e}nyi graph
and the random geometric graph~\cite{Avi08:rdg}. This graph is shown to have several
good properties, e.g., the small diameter and high clustering
coefficient et al. Hybrid model composed of a power law graph and a
grid-like local graph is studied by several groups of researchers,
see~\cite{CL04:hybrid,KB10:geocommunity,FG09:powerlaw}. Clusters or
communities based on the concept of conductance was studied
in~\cite{KVV04:clustering} and
~\cite{LLDM08:community,LLM10:empirical}, in which the spectral
algorithms and other approximation algorithms were used to detect
good clusters or communities.

 In Section~\ref{sec:def}, we will
introduce the definition of the small-community phenomenon as well
as our models, and then state the main results of the paper. In the next
three sections, we show that the models have the desired properties.
In Section~\ref{sec:parameter}, we discuss the effect of the choice
of a parameter on the properties of our proposed models. Finally, we
give a brief conclusion in Section~\ref{sec:conclu}.

\section{basic definitions, the model and main results}\label{sec:def}
\subsection{The small-community Phenomenon}
In a graph $G=(V,E)$, the degree of a node $v\in V$ is denoted as
$\deg_G(v)$. The volume of a subset of $S\subseteq V$ is defined to
be the sum of degrees of vertices in it, namely, $\vol(S)=\sum_{v\in
S} \deg_G(v)$.

Our definition of communities is inspired by the work of Leskovec
et~al.~\cite{LLDM08:community}, who used the conductance to measure
the goodness of a community. We introduced the concept of $(\alpha,
\beta, \gamma)$-community based on the conductance and the size of a
set of nodes~\cite{LP11:smallcommunity}.  The conductance $\mathrm\Phi(S)$
of $S$ is the ratio between the number of edges coming out of $S$
and the volume of it or its complement $\bar{S}$, whichever is
smaller, i.e.,
\begin{displaymath}
\mathrm \Phi(S) = \frac{|e(S,\bar{S})|}{\min\{\vol(S),\vol(\bar{S})\}}\enspace,
\end{displaymath}
where $e(S,T)$ denotes the set of edges with one endpoint in $S$ and
the other in $T$.

Now we formulate the $(\alpha, \beta, \gamma)$-community as follows:
\begin{definition}
Given a graph $G = (V, E)$ with $|V|=n$, a connected set $S \subset
V$ with $|S|=\omega(1)$ is a strong $(\alpha,\beta)$-community if
\begin{eqnarray}
\mathrm \Phi(S)\leq\frac{\alpha}{|S|^{\beta}}\enspace. \label{eqn:definition}
\end{eqnarray}

Moreover, if $|S|= O((\ln n)^{\gamma})$, then we say that $S$ is
a strong $(\alpha, \beta, \gamma)$-community.
\end{definition}

Note that in the above definition we require that the size of a
community is not too small (i.e., $|S|=\omega(1)$). This requirement
helps us to avoid the trivial case in our definition (when $|S|$ is
constant, it can always be treated as a proper community by choosing
large $\alpha$). In fact, a meaningful community in society always
can not be too small because of lack of requisite variety or other
group function~\cite{All04:alacrity}.

To characterize the feature that almost every one in the network
belongs to some small community, we give the following definition.
\begin{definition}
A network (model) $G$ is said to exhibit the small-community
phenomenon, if almost every\footnote{\textit{almost every} means
$1-o_n(1)$, where $n$ is the number of vertices in $G$} node belongs
to some $(\alpha, \beta, \gamma)$-community, where $\alpha, \beta,
\gamma>0$ are some global constants.
\end{definition}

\subsection{The Geometric Model}~\label{subsec:model}
The base model we will use is a geometric preferential attachment
model with additive fitness. Such a model has been studied
in~\cite{Esk08:geometric,Jor10:gpa} (see
also~\cite{FFV07:gpamodel1,FFV07:gpamodel2}). Assume that a
self-loop counts as degree $1$. The model is defined on a unit-area
spherical surface $S$ (i.e., the radius of the sphere is
$\frac{1}{2\sqrt{\pi}}$). Let $n$ be the number of vertices we are
going to generate. Let $\xi>0$ be an arbitrary constant and $m, r,
\delta=\xi m$ be some parameters which may depend on $n$ (Note that
this is the essential difference from the cases studied
in~\cite{Jor10:gpa}). Intuitively speaking, $m$ is the number of
edges we are going to add in each step; $r$ is the distance
restriction on the two endpoints of an edge; $\delta$ is the
additive fitness. Let $B_R(v)$ denote the spherical cap of radius
$R$ around $v$ in $S$, i.e., $B_R(v)=\{u\in S: \|u-v\|\leq R\}$,
where $\|\cdot\|$ denotes the angular distance on $S$. Let
$A_R=\area(B_R(v))$ be the area of the spherical cap of radius $R$,
which is independent of $v$.

\textbf{The base model}: We start the process from a graph $G_1$,
which is composed of a uniformly generated (from $S$) node $x_1$
with $2m$ self-loops. At each time $t+1$ for $t>0$, if $G_t=(V_t,
E_t)$, we first generate a new node $x_{t+1}$ uniformly at random
from $S$ and then connect it to some existing vertices or itself.
Specifically, if there is no node in $B_r(x_{t+1})$, then we add
$2m$ self-loops to $x_{t+1}$; if $B_r(x_{t+1})\cap V_t\neq
\emptyset$, then we choose independently $m$ contacts (with
replacement) from $B_r(x_{t+1})$ for the new comer such that for any $i$ with
$1\leq i\leq m$, the probability that some vertex $v\in
B_r(x_{t+1})$ is chosen as the $i$th contact is defined by
\begin{eqnarray}
\Pr[y_i^{t+1}=v]=\frac{\deg_t(v)+\delta}{\sum_{w\in B_r(x_{t+1})\cap
V_t}(\deg_t(w)+\delta)} \enspace. \label{eqn:prob}
\end{eqnarray}

Remark: in~\cite{Esk08:geometric} (also
in~\cite{FFV07:gpamodel1,FFV07:gpamodel2}), a self-loop parameter
$\alpha>2$ was introduced to avoid a technical problem when proving
the power law degree distribution. In their settings, a node $v\in
B_r(x_{t+1})$ is chosen as the contact with probability
\begin{eqnarray}
\frac{\deg_t(v)+\delta}{\max\{\sum_{w\in B_r(x_{t+1})\cap
V_t}(\deg_t(w)+\delta),\alpha (m+\delta/2)A_rt\}}\enspace,
\end{eqnarray}
where $\delta>-m$. The case of $\alpha=0$ is left open in these
papers. Jordan~\cite{Jor10:gpa} investigated the asymptotic
behavior of the degree sequence in the case of $\alpha=0$. In his study, $m, r, \delta>0$
are constants that not depend on $n$, which converges to infinity.
However, in our situation, we need a strong concentration result
such that the parameters may
depend on $n$. We will give such a result when $\alpha=0$ and $\delta>0$, which strengths the results in~\cite{Esk08:geometric,Jor10:gpa} and partially answers the open question in~\cite{FFV07:gpamodel1,FFV07:gpamodel2}.

We can show that when $\delta=\xi m>0$
and $r=r_0=n^{-\frac{1}{2}}(\ln n)^{c_0}$, where $c_0=c_0(\xi)$ is
large and may depend on $\xi$, the base model has the power law degree distribution
and the small-community phenomenon but does not have the small diameter. To
incorporate the missing property while not changing the other two
properties too much, we introduce some operations that essentially
generate a uniform recursive tree. We give two different operations
such that the resulted two variants of the base model both have the
three properties to some extent.

\begin{enumerate}
\item \textbf{The hybrid model}:
In this model, every edge has an attribute that indicates whether it
is a {\it local-edge} or a {\it long-edge}, which indicates that the
two endpoints of the edge are local- or long-contacts of each other.
A local- (or long-) edge contributes to the local- (or long-) degree
of both of its endpoints. We start from the $G_1^\mathrm H$ the same as $G_1$ in the above and let
the self-loops of $x_1$ be local-edges. At each step $t+1$ for
$t\geq 1$, to form $G_{t+1}^\mathrm H$ from $G_t^\mathrm H$, a new vertex $x_{t+1}$
is chosen uniformly at random from $S$. First we choose for the new
comer $m$ local-contacts $y_i, 1\leq i\leq m$, independently at
random as in the base model with $\deg_t(v)$ in Eq.
(\ref{eqn:prob}) denoting the \textit{local-degree} of $v$ at time $t$. Then
we choose for $x_{t+1}$ one other long contact $z$ uniformly from
$x_1,\cdots,x_{t}$.

This model can be seen as composed of two parts: a local power law
graph and a global uniform recursive tree, which can be generated in
two phases: firstly, we  can generate the local power law graph
following the rules used in the base geometric model and then
generate a recursive tree as follows: sequentially for $t\geq 1$,
$x_{t+1}$ connects a long-contact which is chosen uniformly at
random from $x_{1},\cdots, x_{t}$.

The independence of the local part and the global part of the hybrid
model conforms to our intuition that local contacts and long
contacts are formed by different mechanisms. Previous studies on
such a model usually has a global power law graph and a local
grid-like graph (see eg.~\cite{CL04:hybrid}), which is comparable
with ours.

\item \textbf{The self-loop model}:
In this model, every new node is born with $\delta$ {\it flexible
self-loops} which may be eliminated in later steps. Now we generate
$x_1$ uniformly at random from $S$ and add $2m+\delta$ self-loops to
it with $\delta\geq 2$ loops marked flexible. This is the start
graph $G_1^\mathrm S$. At each step $t+1$ for $t\geq 1$, to form $G_{t+1}^\mathrm S$
from $G_t^\mathrm S$, a new vertex $x_{t+1}$ is chosen uniformly at random
from $S$ and $\delta$ flexible self-loops are added to it. We first
choose $m$ contacts $y_i, 1\leq i\leq m$ independently at random as
in the base model with $\deg_t(v)$ in Eq. (\ref{eqn:prob}) denoting
\textit{the number of non-flexible edges} incident to $v$ at time $t$. Then
we choose for $x_{t+1}$ one other contact $z$ uniformly from the set
of existing nodes containing flexible self-loop(s)(such a set cannot
be empty because $x_t$ is a member of it) and delete one flexible
self-loop from both $x_{t+1}$ and $z$. The newly added edge
$(x_{t+1}, z)$ is marked flexible. Note that the edge-rewiring keeps
the degree of vertices unchanged, which facilitates the analysis of
its degree distribution.

This model can be seen as composed of two parts: a flexible part and
a non-flexible part, which can be generated in several phases: we
first generate the non-flexible part following the growth rules of
the base model. We then add $\delta$ flexible self-loops to each
vertex. Then sequentially for each $t\geq 1$, $x_{t+1}$ connects a
contact $z$ which is chosen uniformly at random from
$x_1,\cdots,x_{t}$, containing flexible self-loop(s), a flexible
self-loop of $x_{t+1}$ and $z$ is deleted and a new flexible edge
$(x_{t+1},z)$ is added.

We give a plausible explanation of the self-loops emerging in this
model. It is widely studied in social sciences that people in our
society have not only evident relationships with others, but some
implicitly one-sided ``parasocial'' interactions with the
celebrities, virtual characters and so on, in which relationship
only one part knows a great deal about the other, but the other does
not~\cite{HR56:para}. Such a relationship can barely be reflected by
the usually used friendship networks, which mainly coins the
two-sided friendship. Our model incorporates the parasocial
relationships as self-loops and the edge-rewiring may be roughly
interpreted as that the long-distance relationship is established at
the expense of its parasocial connections.
\end{enumerate}

\subsection{Main Results}
Our main results are that the two models have rather good
properties. Assume that $\delta=m\xi$, where $\xi>0$ is some
constant and $r_0=n^{-1/2}(\ln n)^{c_0}$ for some large constant
$c_0$ which may depend on $\xi$.

For $r\geq r_0$, it is obvious that the diameter of the base model is $\Omega(1/r)=\Omega(n^{1/2}(\ln n)^{-c_0})$ (see Section~\ref{sec:diameter}), which is large, while the short diameters of the uniform recursive trees imply the small
diameter results in our two generalized models.
\begin{theorem}(Small Diameter Property)
\label{thm:diameter}
\begin{enumerate}
\item For any $m\geq 1, r>0$, with high probability, the diameter of $G_n^\mathrm H$ is $O(\ln
n)$.
\item For $m\geq K_1(\xi)\ln n$ and $r>0$, with high probability,
the diameter of $G_n^\mathrm S$ is $O(\ln n)$, where $K_1(\xi)$ is some constant depending on
$\xi$.
\end{enumerate}
\end{theorem}

By the geometric structure of the models, it is natural to think of
that a group of vertices close to each other behaves like a good
community. We will make this intuition rigorous by considering the
\textit{$R$-neighborhood $C_R(v)$} of a vertex $v$, which is the set
of all vertices within distance at most $R$ from $v$ in $G_n$ and
show that for some appropriate $r$ and $R$, $C_R(v)$ is a good
community for every $v$.  We give the following result:

\begin{theorem}(Small-Community Phenomenon)
\label{thm:community} If $r=r_0$ and $m\geq K_2(\xi)\ln n$, where
$K_2(\xi)$ is some constant depending on $\xi$, both $G_n^\mathrm H$ and
$G_n^\mathrm S$ have the small-community phenomenon, i.e., in each model,
with high probability, for every node $v\in V_n$, there exists some
$(\alpha,\beta,\gamma)$-community containing $v$, where
$\alpha,\beta,\gamma$ are some constants independent of $n$.
\end{theorem}

A simple corollary of the above theorem is that the base model $G_n$ also has
the small-community phenomenon, which indicates that the community structure is mainly determined by the geometric structure of our model and that the effect of long edges is little for the reason that every new node can establish $m\gg\ln n$ local edges while only $1$ long edges.

The power law degree distribution stems from the preferential
attachment scheme used in our base model, for which we have:

\begin{theorem}(Degree Distribution of the Base Model)
\label{thm:basedeg}
In the base model, if $r\geq r_0$, $m=O(\ln^2 n)$ and $\delta=m\xi$ for any constant
$\xi>0$, there exist some constants $C_k$ and $\mu$, such that for
all $k=k(n)\geq m$,
\begin{eqnarray}
\E[d_k(t)]=C_k\frac{n}{k^{3+\xi}}+O(\frac{n}{(nr^2)^\mu})\enspace,
\end{eqnarray}
where $d_k(t)$ denotes the number of vertices with degree $k$ in the
base model $G_t$, $C_k=C_k(m,\xi)$ tends to a limit
$C_{\infty}(m,\delta)$ which only depends on $m,\delta$ as $k\to
\infty$, and $\mu$ is some constant depending on $\xi$ and strictly less than $1$.
\end{theorem}

Theorem~\ref{thm:basedeg} has already significantly strengthened the results in both
van den Esker~\cite{Esk08:geometric} and Jordan~\cite{Jor10:gpa}. The proof of this theorem
requires the new technique of recursively bounding the concentration
inequalities as we will build in Section~\ref{sec:power}.

Based on Theorem~\ref{thm:basedeg}, we are able to show that in our generalized
models, the networks satisfy a nice power law degree distribution.

\begin{theorem}(Power Law Degree Distribution)
\label{thm:power} For $r\geq r_0$ and $m=O(\ln^2 n)$, the expected degree
sequences of \textbf{the local graph} of the hybrid model $G_n^H$
and \textbf{the whole graph} of the self-loop model $G_n^S$ both
follow a power law distribution with exponent $3+\xi$. More
specifically, there exist some constants $C_k^\mathrm H$, $C_k^\mathrm S$ and $\mu$, such that
for all $k=k(n)\geq m$,
\begin{enumerate}
\item in the hybrid model, $\E[d_k(n)]=C_k^\mathrm H\frac{n}{k^{3+\xi}}+O(\frac{n}{(nr^2)^\mu})$, where $d_k(t)$ denotes the number of vertices with local-degree $k$ in $G_t^\mathrm H$;
\item in the self-loop model, $\E[d_k(n)]=C_k^\mathrm S\frac{n}{k^{3+\xi}}+O(\frac{n}{(nr^2)^\mu})$, where $d_k(t)$ denotes the number of vertices with total degree $k$ in $G_t^\mathrm S$.
\end{enumerate}
In the above statements, both $C_k^\mathrm H$ and $C_k^\mathrm S$ tend to some limits that depend on $m,\delta$ only as $k\to
\infty$, and $\mu$ is some constant depending on $\xi$ and strictly less than $1$.
\end{theorem}

From the above theorems, we know that when $r=r_0=n^{-1/2}(\ln
n)^{c_0}$, the two generalized models simultaneously have all the
three properties to some extent (as in the hybrid model, only the
local part has the power law degree distribution). What about the
cases when $r$ is too large or too small? We give some evidence that
at least one of the three properties disappears in such cases. In
particular, when $r$ is large, we have the following new phenomenon.

\begin{theorem}(Large Community and Small Expander)~\label{thm:parameter}
In the base model $G_n$, let $r=n^{-1/2+\epsilon}$, where $\epsilon>0$ and $m\geq K\ln n$,
for some sufficiently large constant $K$.
\begin{enumerate}
 \item~\label{thm:largecom} If $R=n^{-1/2+\rho}$, for any $\rho>\epsilon$, then
$|C_R(v)|=\Theta(n^{2\rho})$ and $\mathrm
\Phi(C_R(v))=O(\frac{1}{n^{\rho-\epsilon}})$, with high probability.
 \item~\label{thm:smallexp} With high probability, for all $R=o(r)$, $\mathrm \Phi(C_R(v))=\Omega(1)$.
\end{enumerate}
\end{theorem}

Theorem~\ref{thm:parameter} indicates that when $r=n^{-1/2+\epsilon}$, there exists
some large community for every node, which may not belong to any
small community for the reason that the most natural candidate,
i.e., the small neighborhood is not a good community. We remark that
Theorem~\ref{thm:parameter} may imply a new phenomenon in networks. It would be
interesting to find some real world networks, in which there is a
large fraction of nodes each of which is contained in both a good
but large community and a small expander. We also note that the two generalized models have the same phenomenon for such a large $r$.

In the remaining sections of the paper, we are devoted to proving
our main results, Theorems~\ref{thm:diameter}, \ref{thm:community}, \ref{thm:basedeg}, \ref{thm:power}, and \ref{thm:parameter}. We will organize the
paper as follows. In Section~\ref{sec:tools}, we introduce some basic tools for
our proof, and basic properties of our network models. In Sections~\ref{sec:diameter}
and~\ref{sec:community}, we prove Theorems~\ref{thm:diameter} and \ref{thm:community}, respectively. In Section~\ref{sec:power}, we
prove Theorems~\ref{thm:basedeg} and~\ref{thm:power}. In Section~\ref{sec:parameter}, we prove Theorem~\ref{thm:parameter}. Finally in
Section~\ref{sec:conclu}, we discuss some further issues following the results in
this paper.

\section{Useful tools and basic facts}~\label{sec:tools}
Before proving the main results, we first give several basic facts
which will be useful in our proofs of the main results.

We will use the following form of the Chernoff bound (see eg.
Theorem 1.1 in~\cite{DP09:concentration}).
\begin{lemma}~\label{lem:chernoff}
If $X_1,\cdots, X_t$ are independently distributed in $[0,1]$ and
$X=\sum_{i=1}^tX_i$, then for $0<\zeta\leq 1$,
\begin{eqnarray}
\Pr[|X-\E[X]|\geq \zeta\E[X]]\leq 2e^{-\frac{\zeta^2\E[X]}{3}}.
\end{eqnarray}
\end{lemma}

The following submartingale concentration inequality will be used
extensively in our proofs (referred to Theorems 2.38 and 2.41
in~\cite{CL06:complex}).
\begin{lemma}\label{lem:submar}
Suppose that $\{X_0,\cdots, X_t\}$ is a sequence of random variables
associated with a filter $\{\mathcal{F}_0,\cdots,\mathcal{F}_t\}$
and $\mathcal{G}$ is some event on the probability space. If for
$1\leq i\leq t$,
\begin{eqnarray}
&&\E[X_i|\mathcal{F}_{i-1},\mathcal{G}]\leq X_{i-1},\nonumber\\
&&\Var[X_i|\mathcal{F}_{i-1},\mathcal{G}]\leq \sigma_i^2,\nonumber\\
&&X_i-\E[X_i|\mathcal{F}_{i-1},\mathcal{G}]\leq M,\nonumber
\end{eqnarray}
where $\sigma_i^2,M$ are non-negative constants. Then we have
\begin{eqnarray}
\Pr[X_t\geq X_0+\lambda]\leq
e^{-\frac{\lambda^2}{2\sum_{i=1}^t\sigma_i^2+M\lambda/3}}+\Pr[\neg\mathcal{G}].
\end{eqnarray}
\end{lemma}
The supermartingale concentration inequality is similar and we omit
it here.

In the following sections, we will use constants $c_0, c_1$ and $c_2$ which may depend on $\xi$ to characterize some bounds. We state here the conditions that the three constants should satisfy.
\begin{eqnarray}
&& (c_0-c_1-1)(1-1/(\xi+2))< c_1< 2(c_0-c_1-1)(1-2/(2+\xi))~\label{eqn:c1condition}\\
&&c_2=c_1\frac{\ln (\xi(1+\xi/2)+1)}{\ln
((7+400/\xi)^2(\xi(1+\xi/2)+1))}\enspace .~\label{eqn:c2condition}
\end{eqnarray}
Note that for fixed $\xi$ we can always choose $c_0$ to be large enough to guarantee that $c_2$ is also large, which will ensure that the bounds we obtain in the proof are good.

In the definition of our base model, a new vertex will create $2m$
self-loops if there is no existing vertex within distance at most
$r$ from it. This rule is made to guarantee that at each step the
degree of the graph grows by $2m$, which facilitates further
analysis. Moreover, in most interesting cases when $r=
r_0=n^{-1/2}(\ln n)^{c_0}$, if $t$ grows as large as
$\tau=O(\frac{n}{(\ln n)^{2c_0-1}})$, then with high probability for
any vertex that comes after time $\tau$, there will be many existing
nodes within distance at most $r$ from it. Therefore we will focus
on the processes that all the \textit{later} comers will choose
existing nodes as neighbors other than creating $2m$ self-loops.

In analyzing the degree sequence of our base model, it is convenient
to compare the chosen probability given in Eq. (\ref{eqn:prob}) with
the traditional case (eg.~\cite{Bo03:scalefree}), in which at each
step $t+1$ an existing vertex $v$ with degree $k$ is chosen with
probability $\frac{k}{2t}$, where $2t$ is the total degree of all
existing vertices. Thus it is natural to consider of using a good
estimation of (\ref{eqn:prob}) for further analysis. In particular,
we would like to have some good bound on the normalized quantity of
the denominator of (\ref{eqn:prob}). Let $T_t(u)$ denote this
quantity, namely, $T_t(u)=\sum_{v\in B_r(u)\cap
V_{t}}(\deg_t(v)+\delta)$. A closely related quantity is
$Z_t(u)=\sum_{v\in B_r(u)\cap V_{t}}1$, which is the number of
vertices in $B_r(u)$ at time $t$. We have several simple facts on
these two quantities.

\begin{lemma}~\label{lem:ttexp}
If $u\in S$ and $t>0$, then the expectation of $T_t(u)$ is $A_r(2m+\delta)t$.
\end{lemma}

\begin{proof}
Note that
\begin{eqnarray}
\E[T_t(u)]&=&\E[\sum_{v\in B_r(u)\cap V_t}(\deg_t(v)+\delta)]=\E[\sum_{v\in V_t}(\deg_t(v)+\delta)1_{v\in B_r(u)}]\nonumber\\
&=&\E[\sum_{v\in V_t}\deg_t(v)1_{v\in B_r(u)}]+\delta A_rt \enspace. \label{eqn:ett}
\end{eqnarray}

The first part of~(\ref{eqn:ett}) is $2A_rmt$ as given in Lemma 1 and 2 in~\cite{FFV07:gpamodel1}, which completes the proof.
\end{proof}

Let $A_r$ denote the area of $B_r(v)$. Then $A_r=\area(B_r(v))\sim
r^2/4$, for $r=o(1)$. Let $t_r=\frac{12(\ln n)^2n^{c_1/c_0}}{r^{2(1-c_1/c_0)}}$ and thus
$A_rt_r\sim3(\ln n)^2(nr^2)^{c_1/c_0}$. We will consider that $r\geq r_0=n^{-1/2}(\ln n)^{c_0}$ and let $t_0:=t_{r_0}=\frac{12n}{(\ln n)^{2c_0-2c_1-2}}$.

We first give an estimation of the quantity $Z_t(u)$.

\begin{lemma}~\label{lem:zt}
If $r\geq r_0$, then for any $t\geq t_r$, with probability at least $1-2n^{-\ln n}$, we
have that
\begin{displaymath}
|Z_t(u)-A_rt|\leq \frac{1}{(nr^2)^{c_1/2c_0}} A_rt.
\end{displaymath}
\end{lemma}

\begin{proof}
Noticing that $Z_t(u)=\sum_{i=1}^t1_{x_i\in B_r(u)}$ and that
$\Pr[1_{x_i\in B_r(u)}=1]=A_r$, we can obtain the result by simply
applying the Chernoff bound.
\end{proof}

From the above lemma, we can give a rough bound on $T_t(u)$.
\begin{lemma}~\label{lem:roughbound}
If $r\geq r_0$, then for any $t\geq t_r$, with probability at least $1-4n^{-\ln n}$, we have
that
\begin{eqnarray}
(1-\frac{1}{(nr^2)^{c_1/2c_0}})(1+\xi)mA_rt\leq T_t(u)\leq 4(1+\frac{1}{(nr^2)^{c_1/2c_0}})(2+\xi)mA_r.
\end{eqnarray}
\end{lemma}

\begin{proof}
The left inequality is obvious by using the trivial relation that
$T_t(u)\geq m(1+\xi)Z_t$ and the bound on $Z_t$ given in
Lemma~\ref{lem:zt}.

To see the right inequality, we note that the sum of the degrees of
vertices in $B_r(u)$ is equal to the sum of out-degrees of all
vertices in $B_r(u)$, which is equal to $mZ_t$, plus the sum of the
in-degrees of vertices in $B_r(u)$, which is at most the sum of
out-degrees of all vertices in $B_{2r}(u)$. Therefore,
$T_t(u)\leq(m+\delta)Z_t+m\sum_{v\in V_t\cap B_{2r}(u)}1\leq
(2m+\delta)\sum_{v\in V_t\cap B_{2r}(u)}1\leq
(2m+\delta)A_{2r}t(1+\frac{1}{(nr^2)^{c_1/2c_0}}) = 4(2+\xi)mA_rt(1+\frac{1}{(nr^2)^{c_1/2c_0}})$, with
probability $1-2n^{-\ln n}$.
\end{proof}

\section{Small Diameter}~\label{sec:diameter}
It is obvious that the diameter in the base model is at least
$\Omega(1/r)=\Omega(n^{1/2}(\ln n)^{-c_0})$ for all $r\geq r_0$,, since any vertex can connect nodes that within
distance at most $r$ from it and the maximum distance of two
vertices is $\Omega(1)$. However, with the addition of the ability
to choose uniformly from the subset of previous vertices, the
diameter can be reduced to $O(\ln n)$, with high probability. We
will use the following classic result on the diameter and the
maximum degree of a uniform recursive tree.

\begin{lemma}~\label{lem:rrt}
With high probability, the diameter and the maximum degree in a
uniform recursive tree is $\Theta(\ln n)$.
\end{lemma}
\begin{proof} This is a classic result, for which a proof is referred to
 such as~\cite{Pit94:rrt,DL95:degree}.
 \end{proof}

 Now the diameter of
the two generalized models can be bounded as follows.

\begin{proof}[Proof of Theorem~\ref{thm:diameter}]
We consider the two models separately.
\begin{enumerate}
\item In the hybrid model, no matter how the local graph grows, the global
graph is the same as the uniform recursive tree, which gives an
upper bound $O(\ln n)$ on the diameter of the whole graph.
\item For the self-loop model, the constructed tree in the flexible part
are restricted to having degree at most $\delta$ and thus may be
different from a uniform recursive tree. However, by
Lemma~\ref{lem:rrt}, the maximum degree of a uniform recursive tree
is $L\ln n$, where $L$ is the hidden constant in $\Theta(\ln n)$,
from which we know that if $\delta\geq L\ln n$, then with high
probability, the constructed tree in the flexible part is the same
as the uniform recursive tree. Therefore, the diameter of the
self-loop model is again upper bounded by $O(\ln n)$. Finally, we
note that $\delta=m\xi\geq L\ln n$ is equivalent to $m\geq L\ln
n/\xi$, which completes the proof.
\end{enumerate}

This completes the proof of Theorem~\ref{thm:diameter}.
\end{proof}

\section{The Small-Community Phenomenon}~\label{sec:community}
In this section, we consider the community structure and we will require that $r=r_0$. We start from
the intuition that a group of people close to each other form a good
community, which can be thought of geographical
communities~(\cite{KB10:geocommunity}). In particular, for a node
$v$, we define the $R$-neighborhood $C_R(v)$ of $v$ to be the set of
vertices within distance at most $R$ from $v$ in $G_n$, i.e.,
$C_R(v)=B_R(v)\cap V_n$. Let $R_0=n^{-1/2}(\ln n)^{2c_0}$, we will
show that $C_{R_0}(v)$ is a good community. In this section, we will
assume that $m\geq K_2(\xi)\ln n$, where $K_2(\xi)$ is some large
constant depending on $\xi$.

Note that given $v$, the probability that a node generated uniformly
at random from $S$ will land in $B_{R_0}(v)$ is $A_{R_0}\sim
R_0^2/4=\frac{(\ln n)^{4c_0}}{4n}$. Using the Chernoff bound, it is
easy to show that with high probability, the number of nodes in
$C_{R_0}$ is $\Theta((\ln n)^{4c_0})$, which means that the size of
such $R$-neighborhood is small. Now we consider the connectivity of
the subgraph induced by $C_{R_0}(v)$.

\begin{lemma}~\label{lem:connectivity}
In the base model, if $r=r_0=n^{-1/2}(\ln n)^{c_0}$, then for any
$v\in V_n$, the $R_0$-neighborhood $C_{R_0}(v)$ induces a connected
subgraph in $G_n$ with high probability.
\end{lemma}

\begin{proof}
We will first show that for every $v$, $C_{r/2}(v)$ induces a
connected subgraph in $G_n$ with high probability. The lemma then
follows from the fact that any two vertices $u, u'$ in $C_{R_0}(v)$
can be connected by a set of paths between vertices
$u=v_1,v_2,\cdots,v_k=u'$ such that each vertex pair $(v_i,v_{i+1})$
is within distance $r/2$.

Now we consider the connectivity of $C_{r/2}(v)$.

Let $A_{r}T=12\ln n$, and thus $T=\frac{12n}{(\ln n)^{2c_0-1}}$. Let
$H_0$ be the subgraph induced by nodes within distance at most $r/2$
from $v$ at time $T$. Now let $x_{t_1},\cdots,x_{t_k}$ be the nodes
that land in $B_{r/2}(v)$ after time $T$ and let $H_s$ be the
corresponding subgraph when vertex $x_{t_s}$ is added in
$B_{r/2}(v)$. Since every vertex $x_j$ will land in $B_{r/2}(v)$
with probability $A_{r/2}$, we know that with high probability, for
$t\geq T$, the number of nodes in $B_{r/2}(v)$ will be in the range
$[\kappa_1A_{r/2}t, \kappa_2A_{r/2}t]$ for some constants
$\kappa_1,\kappa_2$. In particular, we have that $|H_0|\leq
\kappa_2A_{r/2}T = 3\kappa_2\ln n$ and
$\kappa_1A_{r/2}t_s\leq|H_s|\leq\kappa_2A_{r/2}t_s$.

Now let $X_s$ be the number of connected components of $H_s$ and let
$Y_s$ be the number of connected components of $H_s$ connected to
$x_{t_{s+1}}$. Then we have
\begin{eqnarray}
X_{s+1}=X_s-Y_s+1, X_0\leq 3\kappa_2\ln n\enspace.\nonumber
\end{eqnarray}

We show that if $s\leq 6\kappa_2\ln n$, $X_s$ decreases by at least
$1$ for every $s\geq 1$ with probability at least $\frac{7}{10}$,
from which we know that the probability that $H_{6\kappa_2\ln n}$ is
not connected is bounded by $O(n^{-3})$ and then the Lemma follows
from the fact that each later coming vertex $x_{t_{s+1}}$ such that
$s\geq6\kappa_2\ln n$ will connect the $H_s$ with probability at
least $1-O(n^{-10})$.

Let $\mathcal{E}$ denote the event that for any $u\in V_n$ and for
each $t\geq T$, $T_t(u)\leq 32(2+\xi)mA_{r/2}t$, then as in the proof
Lemma~\ref{lem:roughbound}, the probability that $\mathcal{E}$ holds
is $1-O(n^{-4})$. Now Conditioned on $\mathcal{E}$, for each $1\leq
s\leq 6\kappa_2\ln n$, since $x_{t_s}$ is in $B_{r/2}(v)$, we have
that $|x_{t_{s}}-u|\leq r$ for every vertex $u\in H_{s-1}$ and thus
$x_{t_{s}}$ will connect $u$ with probability at least
\begin{eqnarray}
\frac{m+\delta}{T_{t_s-1}(x_{t_s})}\geq
\frac{1}{32A_{r/2}t_s}\enspace . \nonumber
\end{eqnarray}

Therefore, the probability that $x_{t_s}$ will not connect any
vertex in $B_{r/2}(v)$ is
\begin{eqnarray}
\Pr[Y_s=0]\leq(1-\frac{|H_s|}{32A_{r/2}t_s})^m\leq n^{-10}\enspace,
\nonumber
\end{eqnarray}
where the last inequality follows from the fact that $m\geq
K_2(\xi)\ln n$.

Now we consider the case that $H_s$ has at least two connected
components, namely, $X_s\geq 2$. The probability that $x_{t_{s+1}}$
will connect at most one component is that
\begin{eqnarray}
\Pr[Y_s=1|X_s\geq 2]\leq 2(1-\frac{1}{32A_{r/2}t_s})^m\leq
1/10\enspace, \nonumber
\end{eqnarray}
where we used the fact that $32A_{r/2}t_s\geq96\ln n$ and that
$m\geq K_2(\xi)\ln n$.

Therefore, $X_s$ decreases by at least $1$ for every $1\leq s\leq
6\kappa_2\ln n$ with probability at least $\frac{7}{10}$, which
completes our proof.
\end{proof}

Now we show that the conductance of $C_{R_0}(v)$ in each model is
small.
\begin{lemma}~\label{lem:conduc}
In both the hybrid model and the self-loop model, with high
probability, for any $v\in V_n$, we have that
\begin{eqnarray}
\mathrm \Phi(C_{R_0}(v))=
O\left(\frac{1}{|C_{R_0}(v)|^{1/4c_0}}\right)\enspace
.~\label{lem:conduc}
\end{eqnarray}
\end{lemma}

\begin{proof}
We first consider the hybrid model. For convenience, we abbreviate
$C_{R_0}(v)$ as $C$. Let $e(C,\bar{C})$ denote the set of edges that
connecting $C$ and its complement. Let $e_{1}(C,\bar{C})$ and
$e_{2}(C,\bar{C})$ denote edges in $e(C,\bar{C})$ that are local and
long, respectively. Then we have: $e(C,\bar{C})=e_1(C,\bar{C})\cup
e_2(C,\bar{C})$.

Local edges connecting $C$ and $\bar{C}$ must lie between the two
spherical segments separated by the boundary of $C_{R_0}(v)$. More
specifically, if $e=(u,w)\in e_1(C,\bar{C})$, then one of $u,w$ lies
on the strip $str_1=B_{R_0+r}(v)\backslash B_{R_0}(v)$ and the other
point lies on the strip $str_2=B_{R_0-r}(v)\backslash B_{R_0}(v)$.
With high probability, the total number of vertices in $str_1$ is at
most $n(2rR_0+r^2)$ and the total number of vertices in $str_2$ is
at most $n(2rR_0-r^2)$. Hence the number of local edges that lies
between the two strips is at most $4mnrR_0$, namely,
$|e_1(C,\bar{C})|\leq 4mnrR_0$.

Now we consider the long edges that connects $C$ and the rest of the
graph. We will show that the number of such edges is relatively
small compared with the local edges therein. More precisely, we have
the following lemma.

\begin{lemma}~\label{lem:longdegree}
In the hybrid model, let $Y_t$ denote the sum of the long-degrees of
vertices in $B_{R_0}(v)\cap V_t$. Then $Y_n\leq cA_{R_0}n$ for some
constant $c$, with high probability.
\end{lemma}

\begin{proof}
By definition, we have the following recurrence for $Y_t$.
\begin{eqnarray}
\E[Y_{t+1}|Y_t]=Y_t+A_{R_0}+\frac{|B_{R_0}(v)\cap V_t|}{t}\enspace.
\end{eqnarray}

Let $A_{R_0}T=12\ln n$, and thus $T=\frac{12n}{(\ln n)^{4c_0-1}}$.
Let $\mathcal{F}$ denote the event that for all $t\geq T$, the
relation $|B_{R_0}(v)\cap V_t|\in [\kappa_1A_{R_0}t,
\kappa_2A_{R_0}t]$ holds for some constants $\kappa_1$ and
$\kappa_2$ and that the maximum long-degree of vertices $x_1,\cdots,
x_T$ is $L\ln n$. By Lemma~\ref{lem:rrt} and the Chernoff bound, we
know that $\Pr[\mathcal{F}]\geq 1-O(n^{-3})$.

Now we know that for $t\geq T$,
\begin{eqnarray}
\E[Y_{t+1}|Y_t,\mathcal{F}]\leq
Y_t+A_{R_0}+\frac{\kappa_2A_{R_0}t}{t}\enspace,\nonumber
\end{eqnarray}
from which we have
\begin{eqnarray}
\E[Y_{t+1}|Y_t,\mathcal{F}]-(1+\kappa_2)A_{R_0}(t+1)\leq
Y_t-(1+\kappa_2)A_{R_0}t\enspace.~\label{eqn:longdegree}
\end{eqnarray}

Conditioned on $\mathcal{F}$, we know that the number of vertices in
$B_{R_0}(v)\cap V_t$ is $\kappa_2A_{R_0}T\leq 12\kappa_2\ln n$ and
every vertex in this set has degree at most $L\ln n$, from which we
know that $Y_T\leq 12\kappa_2L(\ln n)^2$. Now define
\begin{displaymath}
X_\tau=\left\{ \begin{array}{ll}
Y_\tau-(1+\kappa_2)A_{R_0}\tau & \textrm{for $\tau\geq T+1$}\enspace, \\
12\kappa_2L(\ln n)^2 & \textrm{for $\tau=T$}\enspace.
\end{array}
\right.
\end{displaymath}

By inequality~(\ref{eqn:longdegree}), $X_T,\cdots, X_t$ forms a
submartingale with error $O(n^{-3})$. We also have that for
$\tau>T$,
\begin{eqnarray}
X_{\tau}-\E[X_{\tau}|X_{\tau-1}]\leq 1\enspace,\nonumber
\end{eqnarray}
and
\begin{eqnarray}
\Var[X_{\tau}|X_{\tau-1}]&=&\Var[Y_{\tau}|X_{\tau-1}]\nonumber\\
&\leq&\E[(Y_\tau-Y_{\tau-1})^2|X_{\tau-1}]\nonumber\\
&\leq&(1+\kappa_2)A_{R_0}\enspace.\nonumber
\end{eqnarray}

Now we apply the submartingale concentration inequality as in
Lemma~\ref{lem:submar}, we have that
\begin{eqnarray}
\Pr[X_t\geq X_T+\lambda]&\leq&
e^{-\frac{\lambda^2}{2(\sum_{\tau=T+1}^t(1+C_2)A_{R_0}+\lambda/3)}}+O(n^{-3})\nonumber\\
&\leq
&e^{-\frac{\lambda^2}{2t(1+C_2)A_{R_0}+2\lambda/3}}+O(n^{-3})\enspace.\nonumber
\end{eqnarray}

Let $\lambda=c'\sqrt{\ln nA_{R_0}t}$ for some constant $c'$. Then
\begin{eqnarray}
\Pr[X_t\geq X_T+c'\sqrt{\ln nA_{R_0}t}]\leq
O(n^{-3})\enspace.\nonumber
\end{eqnarray}

Finally, using $X_t=Y_t-(1+\kappa_2)A_{R_0}t$, we have
\begin{eqnarray}
\Pr[Y_t\geq (1+\kappa_2)A_{R_0}t+c'\sqrt{\ln
nA_{R_0}t}+12\kappa_2L(\ln n)^2]\leq O(n^{-3})\enspace.\nonumber
\end{eqnarray}

In particular, with high probability $Y_n\leq cA_{R_0}n$ for some
constant $c$, which completes the proof.
\end{proof}

By Lemma 9, we know that $|e_2(C,\bar{C})|\leq cA_{R_0}n$.

Thus, the total number of edges between $C$ and $\bar{C}$ is
\begin{eqnarray}
|e(C,\bar{C})|= O(mrR_0n+R_0^2n)\enspace.
\end{eqnarray}

The volume of $C$ is at least $m|C|\sim mR_0^2n$, which means that
\begin{eqnarray}
\mathrm \Phi(C)=O\left(\frac{m4rR_0n+R_0^2n}{mR_0^2n}\right)=O((\ln
n)^{-1})=O\left(\frac{1}{|C|^{1/(4c_0)}}\right) \enspace.
~\label{eqn:conduc}
\end{eqnarray}

Finally, we briefly discuss the case in the self-loop model. Let
$\delta \geq K_1(\xi) \ln n$, then with high probability, the
constructed tree in the flexible part of the model is a uniform
recursive tree as the same as that in the proof of
Theorem~\ref{thm:diameter}. Therefore, the edges that connect an
$R$-neighborhood and its complement can be also bounded by the same
argument as that in the case of the hybrid model, which then gives
the same result as~(\ref{eqn:conduc}).
\end{proof}

Now we can show that the two models have the small-community
phenomenon.
\begin{proof}[Proof of Theorem~\ref{thm:community}]
For each $v\in V_n$, the $R_0$-neighborhood $C_{R_0}(v)$ is of size
$\Theta((\ln n)^{4c_0})$. By Lemmas 7 and 8, we know that
$C_{R_0}(v)$ is an $(\alpha, \beta, \gamma)$-community of $v$, where
$\alpha$ is the hidden constant in term
$O\left(\frac{1}{|C_{R_0}(v)|^{1/4c_0}}\right)$ in
Eq.~(\ref{eqn:conduc}), $\beta=1/4c_0$, and $ \gamma=4c_0$. This
completes the proof of Theorem~\ref{thm:community}.
\end{proof}

Note that the proof of Theorem~\ref{thm:community} also implies that the base model $G_t$ has the small-community phenomenon. In fact, in this case, we do not need to consider the effect of the edges generated in the uniform recursive tree, which simplifies the analysis. We can easily show that the $R_0$-neighborhood $C_{R_0}(v)$ has small size, induces a connected subgraph and has conductance $\Phi(C_{R_0}(v))=O(\frac{m4rR_0n}{mR_0^2n})=O((\ln n)^{-c_0})=O(\frac{1}{|C_{R_0}(v)|^{1/4}})\leq\frac{\alpha'}{|C_{R_0}(v)|^{1/4}}$, i.e., every node in the base model is contained in a $(\alpha', 1/4,4c_0)$-community.

\section{The Power Law Degree Distribution}~\label{sec:power}
In this section we prove Theorems~\ref{thm:basedeg}
and~\ref{thm:power}. In Subsection~\ref{subsec:basedeg}, we prove
Theorem~\ref{thm:basedeg} by assuming a concentration inequality of
the degree sequence, in Subsection~\ref{subsec:alterconcen}, we
develop an {\it alternating concentration method} to prove the
concentration inequality desired, and in
Subsection~\ref{subsec:power}, we prove Theorem~\ref{thm:power}.

\subsection{The degree sequence on the base model}
\label{subsec:basedeg} To prove Theorem~\ref{thm:basedeg}, we
analyze a recurrence on $\E[d_k(t)]$ as usual. Recall that
$T_t(u)=\sum_{v\in B_r(u)\cap V_t}(\deg_t(v)+\delta)$. As mentioned
above, we will first give a good estimation of $T_t(u)$ and show
that $T_t(u)$ concentrates around its expectation, building on which
we can derive the degree sequence from the recurrence on
$\E[d_k(t)]$.


Recall that $t_r=\frac{12(\ln n)^2n^{c_1/c_0}}{r^{2(1-c_1/c_0)}}$ for any $r\geq r_0$. We have the following concentration
inequality of $T_t(u)$:
\begin{lemma}~\label{lem:exabound}(Alternating Concentration Theorem)
If $r\geq r_0$, then for all $t\geq t_r$, we have that
\begin{eqnarray}
\Pr[|T_t(u)-(2+\xi)mA_rt|\geq \frac{1}{(nr^2)^{c_2/2c_0}}mA_rt]=
O(n^{-2})\enspace,\label{eqn:tt}
\end{eqnarray}
where $c_1, c_2$ are some constants satisfying the conditions in Eq.~(\ref{eqn:c1condition}) and
(\ref{eqn:c2condition}).
\end{lemma}

Lemma~\ref{lem:exabound} is one of our key technical contributions
in this paper which is interesting by its own. To prove it, we will
need to develop an {\it alternating concentration method}, allowing
us to alternatively and iteratively apply both the submartingale and
supermartingale inequalities to prove a desired concentration
result. The full proof of Lemma 10 is given in
Subsection~\ref{subsec:alterconcen}.

The role of Lemma~\ref{lem:exabound} is to give a good estimation of
$\E\left[\frac{1_{|x_{t+1}-v|\leq r}}{T_t(x_{t+1})}|G_t\right]$ to
analyze the recurrence of $\E[d_k(t)]$. In this subsection, we prove
Theorem~\ref{thm:basedeg} by assuming Lemma~\ref{lem:exabound} as
follows.

\begin{proof}[Proof of Theorem~\ref{thm:basedeg}]
Define $D_k(t):=\{v\in V(G_t)|\deg_{G_t}(v)=k\}$. Then
$d_k(t)=|D_k(t)|$.

The recurrence for the expectation of $d_k(t)$ can be written as
follows.
\begin{eqnarray}
&&\E[d_k(t+1)|G_t]\nonumber\\
& = & d_k(t)+\sum_{v\in D_{k-1}(t)}
\left(m\E\left[\frac{(k-1+\delta)1_{|x_{t+1}-v|\leq r}}{T_t(x_{t+1})}|G_t\right]\right)\nonumber \\
& & -\sum_{v\in
D_{k}(t)}\left(m\E\left[\frac{(k+\delta)1_{|x_{t+1}-v|\leq
r}}{T_t(x_{t+1})}|G_t\right]\right)+O(m\E[\eta_k(G_t,x_{t+1})|G_t])
\enspace, \label{eqn:dkrecur}
\end{eqnarray}
where $\eta_k(G_t,x_{t+1})$ denotes the probability that a parallel
edge from the new vertex $x_{t+1}$ to a vertex of degree no more
than $k$ is created, which is at most
\begin{eqnarray}
{m \choose 2}\sum_{i=m}^k\sum_{v\in
D_i(t)}(i+\delta)^2\left(\frac{1_{|v-x_{t+1}|\leq
r}}{T_t(x_{t+1})}\right)^2 \enspace.\nonumber
\end{eqnarray}

Now for $t\geq t_r$, let $\mathcal{A}_t$ denote the event that
$|T_t(u)-(2+\xi)mA_rt|\leq \frac{1}{(nr^2)^{c_2/2c_0}}mA_rt$. By
Lemma~\ref{lem:exabound}, we have
\begin{eqnarray}
\Pr[\mathcal{A}_t]=1-O(n^{-2})\enspace. \nonumber
\end{eqnarray}

Therefore for $t\geq t_r$,
\begin{eqnarray}
&&\E\left[\sum_{v\in D_k(t)}\frac{(k+\delta)1_{|x_{t+1}-v|\leq r}}{T_t(x_{t+1})}\right]\nonumber\\
&=&\E\left[\sum_{v\in D_k(t)}\frac{(k+\delta)1_{|x_{t+1}-v|\leq
r}}{(2+\xi)mA_rt}\left(1+O(\frac{1}{(nr^2)^{c_2/2c_0}})\right)|\mathcal{A}_t\right]\Pr[\mathcal{A}_t]+O(n^{-2})\nonumber\\
&=&\frac{(k+\delta)}{(2+\xi)mt}\left(1+O(\frac{1}{(nr^2)^{c_2/2c_0}})\right)\E[d_k(t)|\mathcal{A}]\Pr[\mathcal{A}]+O(n^{-2})\nonumber\\
&=&\frac{(k+\delta)}{(2+\xi)mt}\left(1+O(\frac{1}{(nr^2)^{c_2/2c_0}})\right)(\E[d_k(t)]-\E[d_k(t)|\neg\mathcal{A}]\Pr[\neg\mathcal{A}])+O(n^{-2})\nonumber\\
&=&\frac{(k+\delta)\E[d_k(t)]}{(2+\xi)mt}+O(\frac{1}{(nr^2)^{c_2/2c_0}})\enspace.\nonumber
\end{eqnarray}

Similarly, we have
\begin{eqnarray}
&&\E\left[\sum_{v\in D_{k-1}(t)}\frac{(k-1+\delta)1_{|x_{t+1}-v|\leq r}}{T_t(x_{t+1})}\right]\nonumber\\
&=&\frac{(k-1+\delta)\E[d_{k-1}(t)]}{(2+\xi)mt}+O(\frac{1}{(nr^2)^{c_2/2c_0}})\enspace.\nonumber
\end{eqnarray}

The error term can be bounded as follows.
\begin{eqnarray}
&&\E[\eta_k(G_t,x_{t+1})]\nonumber\\
&\leq& {m \choose 2}\E\left[\sum_{i=m}^k\sum_{v\in D_i(t)}(k+\delta)^2\left(\frac{1_{|v-x_{t+1}|\leq r}}{T_t(x_{t+1})}\right)^2\right].\nonumber\\
&\leq& {m \choose 2}\E\left[\sum_{i=m}^k\sum_{v\in D_i(t)}(k+\delta)^2\frac{1}{m^2A_rt^2}\left(1+O(\frac{1}{(nr^2)^{c_2/2c_0}})\right)\right]+O(n^{-2})\nonumber\\
&\leq& O\left(\frac{(k+\delta)^2}{A_rt}\right)+O(n^{-2})
\enspace.\nonumber
\end{eqnarray}

If $k+\delta\leq k_0(t)=(nr^2)^{c_1/2c_0-c_2/4c_0}$, then
$\E[\eta_k(G_t,x_{t+1})]=O(\frac{1}{(\ln n)^{2}(nr^2)^{c_2/2c_0}})$
and $\E[m\eta_k(G_t,x_{t+1})]=O(\frac{1}{(nr^2)^{c_2/2c_0}})$
given the fact that $m=O(\ln^2 n)$.

Let $\bar{d}_k(t):=\E[d_k(t)]$. Now the recurrence can be simplified
as
\begin{eqnarray}
\bar{d}_{k}(t+1)&=&\bar{d}_k(t)-\frac{(k+\delta)\bar{d}_k(t)}{(2+\xi)t}
+\frac{(k-1+\delta)\bar{d}_{k-1}(t)}{(2+\xi)t}\nonumber\\
&&+1_{k=m}+O(\frac{1}{(nr^2)^{c_2/2c_0}})
\enspace.\label{eqn:dk}
\end{eqnarray}

We define a new recurrence related to~(\ref{eqn:dk}). For $j<m$, let
$f_j=0$ and for $j\geq m$, let
\begin{eqnarray}
f_k&=&
\frac{k-1+\delta}{2+\xi}f_{k-1}-\frac{k+\delta}{2+\xi}f_k+1_{d=m}
\enspace,\label{eqn:fk}
\end{eqnarray}
which has solution that $f_m=\frac{2+\xi}{2+\xi+m+\delta}$ and for
$k\geq m+1$,
\begin{eqnarray}
f_k&=&\prod_{j=m+1}^k\frac{j-1+\delta}{2+\xi+j+\delta}f_{m}\nonumber\\
&=&\frac{\Gamma(k+\delta)\Gamma(m+4+\xi+\delta)}{\Gamma(3+\xi+k+\delta)\Gamma(m+1+\delta)}\frac{2+\xi}{2+\xi+m+\delta}\nonumber\\
&=&\frac{\phi_k(m,\delta)}{k^{3+\xi}}\enspace,\nonumber
\end{eqnarray}
where $\phi_k(m,\delta)$ tends to a limit $\phi_{\infty}(m,\delta)$
which depends on $m,\delta$ only as $k\to\infty$.

Now we show that
\begin{eqnarray}
|\bar{d}_k(t)-f_kt|\leq M\left(t_r+\frac{n+Lt}{(nr^2)^{c_2/2c_0}}\right)\enspace.\label{ineqn:degree}
\end{eqnarray}
where $M$ is some large constant and $L$ is the hidden constant in term $O(\frac{1}{(nr^2)^{c_2/2c_0}})$ in Eq.~(\ref{eqn:dk}).

We prove (\ref{ineqn:degree}) by induction.
\begin{enumerate}
\item First note that for $t\leq t_r$, the above relation holds trivially; for $t\geq t_r$ and $k\geq k_0(t)$, the inequality follows from the fact that $\bar{d}_k(t)\leq 2mt/k$.
\item Now assume that $t\geq t_r$ and $k\leq k_0(t)$, we have
\begin{eqnarray}
&&|\bar{d}(k+1)-f_k(t+1)|\nonumber\\
&=& |\bar{d}_k(t)-f_k(t+1)-\frac{(k+\delta)\bar{d}_k(t)}{(2+\xi)t}
+\frac{(k-1+\delta)\bar{d}_{k-1}(t)}{(2+\xi)t}+O(\frac{1}{(nr^2)^{c_2/2c_0}})|\nonumber\\
&=& |\bar{d}_k(t)-f_kt-\left(\frac{k-1+\delta}{2+\xi}f_{k-1}-\frac{k+\delta}{2+\xi}f_k\right)\nonumber\\
&&-\frac{(k+\delta)\bar{d}_k(t)}{(2+\xi)t}
+\frac{(k-1+\delta)\bar{d}_{k-1}(t)}{(2+\xi)t}+O(\frac{1}{(nr^2)^{c_2/2c_0}})|\nonumber\\
&\leq &\left(1-\frac{k+\delta}{(2+\xi)t}\right)|\bar{d}_k(t)-f_kt|+ \frac{k-1+\delta}{(2+\xi)t}|d_{k-1}(t)-f_{k-1}t|+O(\frac{1}{(nr^2)^{c_2/2c_0}})\nonumber\\
&\leq & M\left(t_r+\frac{n+Lt}{(nr^2)^{c_2/2c_0}}\right)+L\frac{1}{(nr^2)^{c_2/2c_0}}\nonumber\\
&\leq & M\left(t_r+\frac{n+L(t+1)}{(nr^2)^{c_2/2c_0}}\right)\enspace.\nonumber
\end{eqnarray}
\end{enumerate}

This completes the induction and the proof of
Theorem~\ref{thm:basedeg}.
\end{proof}

\subsection{Estimation of $T_t(u)$ - Alternating Concentration Analysis}
\label{subsec:alterconcen}

In this subsection, we prove the Alternating Concentration Theorem, that is, Lemma~\ref{lem:exabound}.

As mentioned above, Flaxman et
al.~\cite{FFV07:gpamodel1,FFV07:gpamodel2} and van den
Esker~\cite{Esk08:geometric} introduced a new parameter $\alpha>2$
to facilitate the analysis and they used the traditional coupling
technique to bound $T_t(u)$. In our settings, we do not use the
additional parameter $\alpha$ and we can still get a nice bound.
Our idea is to develop a refined
method based on the recurrence directly implied in the definition of
$T_t(u)$. By using this recurrence, we can start from the weak bound
as given in Lemma~\ref{lem:roughbound}, and iteratively improve both
the upper bound and the lower bound of $T_t(u)$. This improvement can be
done by using the submartingale and supermartingale concentration
inequalities as in the proof of Lemma~\ref{lem:longdegree}. This
allows us to show that the accumulated error in the whole process is
small and therefore guarantees the desired bound.

At first, we show that a lower bound can be achieved from a rough
lower bound on $T_t(u)$.
\begin{lemma}~\label{lem:ltoubound}
Fix $r\geq r_0$. If for any $t\geq t_r$,
\begin{eqnarray}
\Pr[T_t(u)\leq (b_l-\frac{r_l}{(nr^2)^{c_1/2c_0}})(2+\xi)mA_rt]\leq
\epsilon_l\enspace,\label{eqn:dt}
\end{eqnarray}
for some $b_l\in [1/2,1)$ and $r_l=o((nr^2)^{c_1/2c_0})$, then for any
$t\geq t_r$,
\begin{eqnarray}
\Pr[T_t(u)\geq (b_u+\frac{r_u}{(nr^2)^{c_1/2c_0}})(2+\xi)mA_rt]\leq
\epsilon_u\enspace,
\end{eqnarray}
where $b_u=\frac{\xi+1}{2+\xi-\frac{1}{b_l}}\in (1,\infty)$,
$r_u=7+40r_l/\xi$, $\epsilon_u=n\epsilon_l+5n^{-\ln n+1}$ and $c_1$ is some constant satisfying the condition given in Eq.
(\ref{eqn:c1condition}).
\end{lemma}

\begin{proof}[Proof of Lemma~\ref{lem:ltoubound}]
We will mainly use the following recurrence.
\begin{eqnarray}
\E[T_{t+1}(u)|G_t]&=&T_t(u)+m(1+\xi)\E[1_{|x_{t+1}-u|\leq r}|G_t]\nonumber\\
&&+\sum_{v\in V_t}m\Pr[y_i^{t+1}=v|G_t]1_{|u-v|\leq
r}\enspace,\nonumber
\end{eqnarray}
where
\begin{eqnarray}
\Pr[y_i^{t+1}=v|G_t]=\E\left[\frac{(\deg_t(v)+\delta)1_{|x_{t+1}-v|\leq
r}}{T_t(x_{t+1})}|G_t\right]\enspace.\nonumber
\end{eqnarray}

Let $\mathcal{G}$ denote the event that for all $t\geq t_r$, the
following inequalities hold: $T_t(u)\geq (b_l-\frac{r_l}{(nr^2)^{c_1/2c_0}})(2+\xi)mA_rt$ and $(1-\frac{1}{(nr^2)^{c_1/2c_0}})(1+\xi)mA_rt\leq T_t(u)\leq 4(2+\xi)mA_r(1+\frac{1}{(nr^2)^{c_1/2c_0}})$. Then by Lemma~\ref{lem:roughbound}
and the bound given in~(\ref{eqn:dt}), $\Pr[\neg\mathcal{G}]\leq n\epsilon_l+4n^{-\ln
n+1}$. Conditioned on $\mathcal{G}$, for $t\geq t_r$, we have
\begin{eqnarray}
\Pr[y_i^{t+1}=v|G_t, \mathcal{G}]&\leq&
\E\left[\frac{(\deg_t(v)+\delta)1_{v\in
B_r(x_{t+1})}}{(b_l-\frac{r_l}{(nr^2)^{c_1/2c_0}})
(2+\xi)mA_rt}|G_t, \mathcal{G}\right]\nonumber\\
&=& \frac{\deg_t(v)+\delta}{(b_l-\frac{r_l}{(nr^2)^{c_1/2c_0}})(2+\xi)mt}\nonumber\\
&\leq& \frac{\deg_t(v)+\delta}{b_l(2+\xi)mt}(1+\frac{4r_l}{(nr^2)^{c_1/2c_0}})\enspace.\nonumber
\end{eqnarray}

Therefore,
\begin{eqnarray}
\E[T_{t+1}(u)|G_t,\mathcal{G}]&\leq&
T_t(u)+m(1+\xi)A_r+\frac{1}{b_l(2+\xi)t}
(1+\frac{4r_l}{(nr^2)^{c_1/2c_0}})T_t(u)\nonumber\\
&\leq& (1+\frac{1}{b_l(2+\xi)t})T_t(u)+(\xi+1+\frac{40r_l}{(nr^2)^{c_1/2c_0}})mA_r\enspace,\nonumber
\end{eqnarray}
where the second inequality uses the rough upper bound on $T_t(u)$
in Lemma~\ref{lem:roughbound}.

Let $b_u=\frac{\xi+1}{2+\xi-\frac{1}{b_l}}$, and $s=40r_l/\xi$, then
\begin{eqnarray}
&&\E[T_{t+1}(u)|G_t,\mathcal{G}]-(b_u+\frac{s}{(nr^2)^{c_1/2c_0}})(2+\xi)mA_r(t+1)\nonumber\\
&\leq &(1+\frac{1}{b_l(2+\xi)t})\Big(T_t(u)-(b_u+\frac{s}{(nr^2)^{c_1/2c_0}})(2+\xi)mA_rt\Big)\nonumber\\
&&+\Big(\frac{b_u}{b_l}+\xi+1-b_u(2+\xi)+(2s+40r_l-s(2+\xi))\frac{1}{(nr^2)^{c_1/2c_0}}\Big)mA_r\nonumber \\
&\leq&(1+\frac{1}{b_l(2+\xi)t})\left(T_t(u)-(b_u+\frac{s}{(nr^2)^{c_1/2c_0}})(2+\xi)mA_rt\right)\enspace.\label{ineqn:tt}
\end{eqnarray}

Now define
\begin{displaymath}
X_i=\left\{ \begin{array}{ll}
\frac{T_i(u)-(b_u+\frac{s}{(nr^2)^{c_1/2c_0}})(2+\xi)mA_ri}{\prod_{j=t_r}^{i-1}(1+\frac{1}{b_l(2+\xi)j})} & \textrm{for $i>t_r$}\enspace ,\\
T_{t_r}(u)-(b_u+\frac{s}{(nr^2)^{c_1/2c_0}})(2+\xi)mA_rt_r & \textrm{for
$i=t_r$}\enspace.
\end{array}
\right.
\end{displaymath}

From inequality~(\ref{ineqn:tt}), we know that
$\E[X_i|G_{i-1},\mathcal{G}]\leq X_{i-1}$ for $t_r<i\leq t$. Let
$\Delta_i=\prod_{j=t_r}^{i}(1+\frac{1}{b_l(2+\xi)j})\sim
(\frac{i}{t_r})^{1/b_l(2+\xi)}$. We have that
\begin{eqnarray}
X_{i}-\E[X_{i}|G_{i-1},\mathcal{G}]=\frac{T_{i}(u)-\E[T_{i}(u)|G_{i-1},\mathcal{G}]}{\Delta_{i-1}}\leq
(2+\xi)m\enspace, \nonumber
\end{eqnarray}
and
\begin{eqnarray}
\Var[X_{i}|G_{i-1},\mathcal{G}]&=&\frac{\Var[T_i(u)|G_{i-1},\mathcal{G}]}{\Delta_{i-1}^2}\nonumber\\
&\leq& \frac{\E[(T_i(u)-T_{i-1}(u))^2|G_{i-1},\mathcal{G}]}{\Delta_{i-1}^2}\nonumber\\
&\leq& (2+\xi)m\frac{\frac{T_{i-1}(u)}{b_l(2+\xi)(i-1)}+(\xi+1+\frac{40r_l}{(nr^2)^{c_1/2c_0}})mA_r}{\Delta_{i-1}^2}\nonumber\\
&\leq& \frac{(\xi+3)^2m^2A_r}{\Delta_{i-1}^2}\enspace.\nonumber
\end{eqnarray}

Therefore, the sequence $X_{t_r},\cdots, X_t$ satisfies the
conditions in Lemma~\ref{lem:submar} with $\Pr[\neg\mathcal{G}]\leq
n\epsilon+4n^{-\ln n+1}$ and
\begin{eqnarray}
&&\sum_{i=t_r+1}^t\Var[X_{i}|G_{i-1},\mathcal{G}]\nonumber\\
&\leq& \sum_{i=t_r+1}^t\frac{(\xi+3)^2m^2A_r}{\Delta_{i-1}^2}\nonumber\\
&\leq& \sum_{i=t_r+1}^t\frac{(\xi+3)^2m^2A_rt_r^{2/b_l(2+\xi)}}{i^{2/b_l(2+\xi)}}\label{eqn:var}\\
&\leq& (\frac{mA_rt_r}{\ln n})^2\enspace.\nonumber
\end{eqnarray}

The last inequality can be seen by using the fact that $A_rt_r\sim3(\ln n)^2(nr^2)^{c_1/c_0}$, the assumption that $(c_0-c_1-1)(1-1/(\xi+2))<
c_1$. Specifically,
\begin{enumerate}
\item if $2/b_l(\xi+2)=1$, then
\begin{eqnarray}
\sum_{i=t_r+1}^t\frac{(\xi+3)^2m^2A_rt_r^{2/b_l(2+\xi)}}{i^{2/b_l(2+\xi)}}&\leq& O(m^2A_rt_r\ln(t/t_r))= O(\frac{m^2A_r^2t_r^2}{A_rt_r/\ln \ln n}) \leq (\frac{mA_rt_r}{\ln n})^2\enspace.\nonumber
\end{eqnarray}
\item if
$2/b_l(\xi+2)>1$, then
\begin{eqnarray}
\sum_{i=t_r+1}^t\frac{(\xi+3)^2m^2A_rt_r^{2/b_l(2+\xi)}}{i^{2/b_l(2+\xi)}}&\leq& O(m^2A_rt_r)= O(\frac{m^2A_r^2t_r^2}{A_rt_r}) \leq (\frac{mA_rt_r}{\ln n})^2\enspace.\nonumber
\end{eqnarray}
\item if $2/b_l(\xi+2)<1$, then
\begin{eqnarray}
&&\sum_{i=t_r+1}^t\frac{(\xi+3)^2m^2A_rt_r^{2/b_l(2+\xi)}}{i^{2/b_l(2+\xi)}}\leq O(m^2A_rt(\frac{t_r}{t})^{2/b_l(\xi+2)})= O(\frac{m^2A_r^2t_r^2}{A_rt_r(\frac{t_r}{t})^{1-2/b_l(\xi+2)}})\nonumber\\
&\leq&\frac{m^2A_r^2t_r^2}{3(\ln n)^{2+2(1-\frac{2}{b_l(\xi+2)})} (nr^2)^{c_1/c_0+(c_1/c_0-1)(1-\frac{2}{b_l(\xi+2)})}}\nonumber\\
&\leq&\frac{m^2A_r^2t_r^2}{3(\ln n)^{2+2(1-\frac{2}{b_l(\xi+2)})} (nr^2)^{c_1/c_0+(c_1/c_0-1)(1-\frac{2}{(\xi+2)})}}\leq (\frac{mA_rt_r}{\ln n})^2\enspace.\nonumber
\end{eqnarray}
\end{enumerate}

If we let $\lambda=2mA_rt_r$, then using the submartingale
concentration inequality, we have
\begin{eqnarray}
&&\Pr[X_t\geq X_{t_r}+\lambda]\nonumber\\
&\leq& e^{-\frac{\lambda^2}{2\sum_{j=t_r+1}^t\Var[X_{i}|G_{i-1},\mathcal{G}]+2(2+\xi)m\lambda/3}}+\Pr[\neg\mathcal{G}]\nonumber\\
&\leq& n\epsilon_l+5n^{-\ln{n}+1}\enspace.\nonumber
\end{eqnarray}

On the other hand, we have that $X_{t_r}\leq 5mA_rt_r$ conditioned on $\mathcal{G}$. Thus, $\Delta_{t-1}(X_{t_r}+\lambda)\leq
7(\frac{t}{t_r})^{1/b_l(2+\xi)}mA_rt_r=7(\frac{t_r}{t})^{1-1/b_l(2+\xi)}mA_rt\leq
\frac{7}{(nr^2)^{c_1/2c_0}}mA_rt$, where the last inequality follows
from the assumption that $(2c_0-2c_1-2)(1-2/(2+\xi))> c_1$.
Therefore,
\begin{eqnarray}
&&\Pr[T_t(u)\geq (b_u+\frac{s+7}{(nr^2)^{c_1/2c_0}})(2+\xi)mA_rt]\nonumber \\
&\leq&\Pr[\frac{T_t(u)-(b_u+\frac{s}{(nr^2)^{c_1/2c_0}})(2+\xi)mA_rt}{\Delta_{t-1}}\geq X_{t_r}+\lambda]\nonumber\\
&\leq&\Pr[X_t\geq X_{t_r}+\lambda]\nonumber\\
&\leq& n\epsilon_l+5n^{-\ln{n}+1}\enspace.\nonumber
\end{eqnarray}

The proof completes by letting $r_u=7+40r_l/\xi$ and
$\epsilon_u=n\epsilon_l+5n^{-\ln{n}+1}$.
\end{proof}

Similarly, from a rough upper bound, we can obtain an upper bound on
$T_t(u)$.
\begin{lemma}~\label{lem:utolbound}
Fix $r\geq r_0$. If for any $t\geq t_r$,
\begin{eqnarray}
\Pr[T_t(u)\geq (b_u+\frac{r_u}{(nr^2)^{c_1/2c_0}})(2+\xi)mA_rt]\leq
\epsilon_u,
\end{eqnarray}
for some $b_u\in (1,4)$ and $r_u=o((nr^2)^{c_1/2c_0})$, then for any $t\geq t_0$,
\begin{eqnarray}
\Pr[T_t(u)\leq (b_l-\frac{r_l}{(nr^2)^{c_1/2c_0}})(2+\xi)mA_rt]\leq
\epsilon_l,
\end{eqnarray}
where $b_l=\frac{\xi+1}{2+\xi-\frac{1}{b_u}}\in (1/2,1)$,
$r_l=7+40r_u/\xi$, $\epsilon_l=n\epsilon_u+5n^{-\ln n+1}$, and
$c_1$ is some constant satisfying the condition given in Eq.
(\ref{eqn:c1condition}).
\end{lemma}

\begin{proof} The proof here is similar to proof of Lemma~\ref{lem:ltoubound}. Note that we should instead use the supermartingale concentration inequality and let $\mathcal{G'}$ denote the good event defined similar to $\mathcal{G}$ in the above proof, which will lead to the following recurrence.
\begin{eqnarray}
\E[T_{t+1}(u)|G_t,\mathcal{G'}]&\geq&
T_t(u)+m(1+\xi)A_r+\frac{1}{b_u(2+\xi)t}
(1-\frac{r_u}{(nr^2)^{c_1/2c_0}})T_t(u)\nonumber\\
&\geq& (1+\frac{1}{b_u(2+\xi)t})T_t(u)+(\xi+1-\frac{5r_u}{(nr^2)^{c_1/2c_0}})mA_r\enspace. \nonumber
\end{eqnarray}

Let $b_l=\frac{\xi+1}{2+\xi-\frac{1}{b_u}}$, and $s'=40r_u/\xi$, then
\begin{eqnarray}
&&\E[T_{t+1}(u)|G_t,\mathcal{G'}]-(b_l-\frac{s'}{(nr^2)^{c_1/2c_0}})(2+\xi)mA_r(t+1)\nonumber\\
&\geq &(1+\frac{1}{b_u(2+\xi)t})\Big(T_t(u)-(b_l-\frac{s'}{(nr^2)^{c_1/2c_0}})(2+\xi)mA_rt\Big)\nonumber\\
&&+\Big(\frac{b_l}{b_u}+\xi+1-b_l(2+\xi)+(-s'-5r_u+s'(2+\xi))\frac{1}{(nr^2)^{c_1/2c_0}}\Big)mA_r\nonumber \\
&\geq&(1+\frac{1}{b_u(2+\xi)t})\left(T_t(u)-(b_l-\frac{s'}{(nr^2)^{c_1/2c_0}})(2+\xi)mA_rt\right)\enspace.\label{ineqn:tt}
\end{eqnarray}

We then define the corresponding supermartingale $X'_{t_r},\cdots, X'_t$ using the above inequality. In this case, we will also use the conditions~(\ref{eqn:c1condition}) on the constants $c_0$ and $c_1$. Then by setting $\lambda'=\frac{29}{4}mA_rt$, where $\lambda'$ corresponds to the parameter $\lambda$ in Lemma~\ref{lem:ltoubound}, and using $X'_{t_r}\geq \frac{1}{4}mA_rt_r$, we will get
\begin{eqnarray}
&&\Pr[T_t(u)\leq (b_l-\frac{s'+7}{(nr^2)^{c_1/2c_0}})(2+\xi)mA_rt]\nonumber \\
&\leq&\Pr[\frac{T_t(u)-(b_l-\frac{s'}{(nr^2)^{c_1/2c_0}})(2+\xi)mA_rt}{\Delta_{t-1}}\leq X'_{t_r}-\lambda']\nonumber\\
&\leq&\Pr[X'_t\leq X'_{t_r}-\lambda']\nonumber\\
&\leq& n\epsilon_u+5n^{-\ln{n}+1}\enspace.\nonumber
\end{eqnarray}
and complete the proof by letting $r_l=7+40r_u/\xi$ and $\epsilon_l=n\epsilon_u+5n^{-\ln n+1}$.
\end{proof}

 Now we are ready to prove
Lemma~\ref{lem:exabound}. Intuitively, we will iteratively apply the
above two lemmas and show that if we start with a rough lower bound
$l_1$, then by Lemma~\ref{lem:ltoubound}, we can get an upper bound
$u$, from which we can again get a new lower bound $l_2$ by
Lemma~\ref{lem:utolbound}. We prove that $l_2>l_1$, which means that
we get a better lower bound in every iteration. The same holds for
the upper bound.

\begin{proof}[Proof of Lemma~\ref{lem:exabound}]
If $\frac{\xi+1}{2+\xi-\frac{2+\xi}{1+\xi}}>4$, then we start our
iterative process from the rough upper bound in
Lemma~\ref{lem:roughbound}. Otherwise, we start the process from the
rough lower bound.

Assume we start from the rough lower bound, and the case of starting
from the rough upper bound is similar. By
Lemma~\ref{lem:roughbound}, we know that for all $t\geq t_r$,
$T_t(u)\geq(1-\frac{1}{(nr^2)^{c_1/2c_0}})(1+\xi)mA_rt$ with probability
at least $1-4n^{-\ln n}$. We define the start point of our iterative
process by letting $b_l^{(1)}=\frac{1+\xi}{2+\xi}\in (1/2, 1)$,
$r_l^{(1)}=1$, $\epsilon_l^{(1)}=5n^{-\ln n}$.

For $i\geq 1$, assume that we have that $T_t(u)\geq
(b_l^{(i)}-\frac{r_l^{(i)}}{(nr^2)^{c_1/2c_0}})(2+\xi)mA_rt$ with error
probability $\epsilon_{l}^{(i)}$ for any $t\geq t_r$. Now we
substitute the corresponding parameters in Lemma~\ref{lem:ltoubound}
to give an upper bound that $T_t(u)\leq
(b_{u}^{(i)}+\frac{r_u^{(i)}}{(nr^2)^{c_1/2c_0}})(2+\xi)mA_rt$ for all
$t\geq t_r$ with error probability $\epsilon_{u}^{(i)}$, where
$b_{u}^{(i)}=\frac{1+\xi}{2+\xi-\frac{1}{b_{l}^{(i)}}}\in (1, 4]$,
$r_{u}^{(i)}=(7+40/\xi)r_l^{(i)}\geq 7+40r_l^{(i)}/\xi$,
$\epsilon_{u}^{(i)}=n\epsilon_{l}^{(i)}+5n^{-\ln n+1}$.

Again we substitute the corresponding parameters in
Lemma~\ref{lem:utolbound} to give an improved lower bound that
$T_t(u)\geq (b_l^{(i+1)}-\frac{r_l^{(i+1)}}{(nr^2)^{c_1/2c_0}})(2+\xi)mA_rt$ for all $t\geq t_r$ with error
$\epsilon_l^{(i+1)}$, where
$b_{l}^{(i+1)}=\frac{1+\xi}{2+\xi-\frac{1}{b_u^{(i)}}}\in (1/2, 1)$,
$r_{l}^{(i+1)}=(7+40/\xi)r_u^{(i)}\geq 7+40r_{u}^{(i)}/\xi$,
$\epsilon_{l}^{(i+1)}=n\epsilon_{u}^{(i)}+5n^{-\ln n+1}$.

Let $C(\xi)=7+40/\xi$. Then $r_{l}^{(i+1)}=C(\xi)^2r_l^{(i)}$ and $\epsilon_{l}^{(i+1)}\leq n^2\epsilon_l^{(i)}+10n^{-\ln n+2}$.

Now we show that for every $i$, $b_{l}^{(i+1)}$ is strictly greater
than $b_{l}^{(i)}$, i.e., the process gives better lower bound after
every  two consecutive steps. Then by the fact that $b_{l}^{(i)}<1$,
we have that $\{b_{l}^{(i)}\}_{i\geq 1}$ converges to $1$.
Similarly, it can be shown that the procedure gives better upper
bound; namely, $\{b_{u}^{(i)}\}_{i\geq 1}$ is a decreasing sequence
which converges to $1$. In the following, we actually prove a
stronger result that after each iteration, the distance between
$b_l^{(i)}$ and $1$ decreases by a multiple factor, which guarantees
that the $\{b_{l}^{(i)}\}_{i\geq 1}$ converges \textit{quickly} to
$1$.

We calculate the distance between $b_{l}^{(i+1)}$ and $1$, which
gives that
\begin{eqnarray}
1-b_{l}^{(i+1)}&=&1-\frac{1+\xi}{2+\xi-\frac{1}{b_{u}^{(i)}}}\nonumber\\
&=&1-\frac{1+\xi}{2+\xi-\frac{1}{\frac{1+\xi}{2+\xi-\frac{1}{b_{l}^{(i)}}}}}\nonumber\\
&=&\frac{1-b_{l}^{(i)}}{\xi(2+\xi)b_{l}^{(i)}+1}\nonumber\\
&\leq&\frac{1-b_{l}^{(i)}}{\xi(1+\xi/2)+1}.\nonumber
\end{eqnarray}

Therefore, the sequence $\{1-b_{l}^{(i)}\}_{i\geq 1}$ decreases by a
multiple factor at least $\frac{1}{\xi(1+\xi/2)+1}$ at each step. On the other hand, since
$T_t(u)\geq
[1-(1-b_l^{(i)})-\frac{r_l^{(i)}}{(\ln n)^{c_1}}](2+\xi)mA_rt$, the best bound is determined by the
maximum of $\frac{r_l^{(i)}}{(\ln n)^{c_1}}$ and $1-b_l^{(i)}$, which is at most $\frac{1/2}{(\xi(1+\xi/2)+1)^i}$. We terminate the iteration at the step $k_0=\lceil\frac{(c_1/2c_0)\ln (nr^2)}{\ln (C(\xi)^2(\xi(1+\xi/2)+1))}\rceil\leq \frac{\ln n}{4}$, in which case $\frac{1/2}{(\xi(1+\xi/2)+1)^{k_0}}\leq\frac{r_l^{(k_0)}}{(nr^2)^{c_1/2c_0}}=\frac{C(\xi)^{2k_0}}{(nr^2)^{c_1/2c_0}}$, and
\begin{eqnarray}
&&\Pr[T_t(u)\leq (1-\frac{1}{(nr^2)^{c_2/2c_0}})(2+\xi)mA_rt]\nonumber\\
&\leq& \Pr[T_t(u)\leq (1-\frac{2C(\xi)^{2k_0}}{(nr^2)^{c_1/2c_0}})(2+\xi)mA_rt]\nonumber\\
&\leq& \epsilon_l^{(k_0)}\leq 2n^{2k_0-\ln n+2}\leq n^{-\ln
n/2+2},\nonumber
\end{eqnarray}
where we used the assumption that $c_2=c_1\frac{\ln
(\xi(1+\xi/2)+1)}{\ln (C(\xi)^2(\xi(1+\xi/2)+1))}$.

The upper bound can be obtained similarly by noting that the sequence $\{b_{u}^{(i)}-1\}_{i\geq 1}$ decreases by a
multiple factor at least $\frac{1}{\xi(2+\xi)+1}\leq\frac{1}{\xi(1+\xi/2)+1}$ at each step. Hence, we have that
\begin{eqnarray}
\Pr[|T_t(u)-(2+\xi)mA_rt|\geq \frac{1}{(nr^2)^{c_2/2c_0}}mA_rt]\leq
n^{-2}.
\end{eqnarray}
\end{proof}

\subsection{Power Law Distribution of the Generalized Models}
\label{subsec:power} In this subsection, we prove
Theorem~\ref{thm:power}, based on both the result and the proof of
Theorem~\ref{thm:basedeg}.

\begin{proof}[Proof of Theorem~\ref{thm:power}]
Since the local-degree sequences in the hybrid model is exactly the
same as the degree sequences in the base model, by
Theorem~\ref{thm:basedeg}, the local graph of $G_n^\mathrm H$ has
the power law degree distribution.

Now for the self-loop model, in which the degree of a node $v$ can
be expressed as $\deg_t(v)+\delta$, where $\deg_t(v)$ is the number
of non-flexible edges incident to $v$ at time $t$. Now we can write
the recurrence as follows.
\begin{eqnarray}
&&\E[d_{k+\delta}(t+1)|G_t]\nonumber\\
& = & d_{k+\delta}(t) + \sum_{v\in D_{k-1+\delta}(t)}\Big(m\E\Big[\frac{(k-1+\delta)1_{|x_{t+1}-v|\leq r}}{T_t(x_{t+1})}|G_t\Big]\Big)\nonumber \\
& & -\sum_{v\in
D_{k+\delta}(t)}\Big(m\E\Big[\frac{(k+\delta)1_{|x_{t+1}-v|\leq
r}}{T_t(x_{t+1})}|G_t\Big]\Big)+O(m\E[\eta_k(G_t,x_{t+1})|G_t]),
\label{eqn:dkrecur}
\end{eqnarray}
Solving the recurrence, we can also arrive at~(\ref{eqn:fk}), which
gives the solution of the form
$\frac{\phi'_k(m,\delta)}{(k+\delta)^{3+\xi}}$, where
$\phi'_k(m,\delta)$ tends to a limit $\phi'_{\infty}(m,\delta)$
which depends only on $m,\delta$ as $k\to\infty$. This finishes the
proof that the degree sequence of the self-loop model follows a
power law distribution.
\end{proof}

\section{Large Community and Small Expander}~\label{sec:parameter}
In this section, we will prove Theorem 5.

Before proving the result, we give a brief discussion on the choice
of $r$. In the previous sections, we considered the case when
$r=n^{-1/2}(\ln n)^{c_0}$ for some sufficiently large constant
$c_0$. The base model as well as the two generalized models has the
small-community phenomenon and the power law degree distribution.
Now we consider other choices of $r$ and show that if $r$ is too
small or too large, then there is a strong evidence indicating that
the model does not have the power law degree distribution or the
small-community phenomenon, respectively.

When $r$ is as small as $r=n^{-1/2-\epsilon}$, for any $\epsilon>0$,
then every node connects only a very small fraction of
neighbors and the whole graph is almost surely disconnected
(\cite{Pen03:rgg}). Furthermore, there are many isolated vertices in the base
model in this range of $r$, which indicates that the base model is
very unlikely to have the power law degree distribution.

When $r$ is as large as $r=n^{-1/2+\epsilon}$, for any $\epsilon>0$,
we have shown that the models have the power law degree distribution. However, the small-community
phenomenon does not seem to exist in this situation. In particular,
there exists an interesting division of the structure of the
$R$-neighborhood when $R$ varies. Specifically, we have shown
in~\cite{LP11:smallcommunity} that under this range of $r$, if
$R=n^{-1/2+\rho}$ for any $\rho>\epsilon$, then with high
probability, for any $v$, $C_R(v)$ is an $(\alpha,\beta)$-community
for some constants $\alpha,\beta$ of size $\Theta(n^{2\rho})$, which
indicates that every node belongs to some large community. Here we
show that with high probability, for all $R=o(r)$, and for any $v\in
V_n$, the conductance $\mathrm \Phi(C_R(v))$ of $C_R(v)$ is larger
than some constant, which indicates that the $R$-neighborhood is not
a good community.

Now we give the proof of Theorem~\ref{thm:parameter}.
\begin{proof}[Proof of Theorem~\ref{thm:parameter}]
The first part of the theorem is given
in~\cite{LP11:smallcommunity}. Here we prove the second part.

For some fixed $R=o(r)$, we let $C=C_R(v)$ and $C'=C_{r-R}(v)$ for
convenience. Then for any vertex $u\in C$ and $u'\in C'$, the
distance between $u$ and $u'$ is at most $r$. The areas of $B_R(v)$
and $B_{r-R}(v)$ are
\begin{eqnarray}
&&\area(B_R(v))\sim R^2/4\nonumber\\
&&\area(B_{r-R}(v))\sim (r-R)^2/4\sim r^2/4, \nonumber
\end{eqnarray}
respectively, which means that a uniformly generated point will land
in $B_R(v)$ and $B_{r-R}(v)$ with probability $R^2/4$ and $r^2/4$,
respectively.

We will show that there are many edges between $C'\backslash C$ and
$C$. To be more specific, let $C_1$ (or $C'_1$) be the vertices in
$C$ (or $C'$) that were born before or at time $n/2$ and $C_2$ (or
$C'_2$) be the set of vertices in $C$ (or $C'$) that were born after
time $n/2$. We show that the sum of the number of edges
$e(C_1,C'_2)$ between $C_1$ and $C'_2$, and the number of edges
$e(C_2,C'_1)$ between $C_2$ and $C'_1$ are large.

Let $\mathcal{E}$ denote the event that for any $u\in V_n$ and for
each $t\geq t_0$, $T_t(u)\leq 8(2+\xi)mA_rt$, then by
Lemma~\ref{lem:roughbound}, the probability that $\mathcal{E}$ holds
is $1-O(n^{-\ln n})$. Now Conditioned on $\mathcal{E}$, for any
vertex $x_j\in C'_2$, the probability that the $i$-th contact of
$x_j$ lies in $C_1$ is at least
$\frac{(m+\delta)|C_1|}{T_{j-1}(x_j)}\geq
\frac{(1+\xi)|C_1|}{4(2+\xi)A_rn}\geq \frac{|C_1|}{8A_rn}$. Thus,
$|e(C_1,C'_2)|$ dominates $Bi(m|C'_2|,\frac{|C_1|}{8A_rn})$, where
$Bi(N,p)$ denotes the binomial distribution with parameters $N$ and
$p$.

Similarly, for any vertex $x_j\in C_2$, the probability that the
$i$-th contact of $x_j$ lies in $C'_1$ is thus at least
$\frac{(m+\delta)|C'_1|}{T_{j-1}(x_j)}\geq
\frac{(1+\xi)|C'_1|}{4(2+\xi)A_rn}\geq \frac{|C'_1|}{8A_rn}$. Thus,
$|e(C_2,C'_1)|$ dominates $Bi(m|C_2|,\frac{|C'_1|}{8A_rn})$.

Totally, the expected number of edges between the $C$ and
$C'\backslash C$ is
\begin{eqnarray}
\E[|e(C,C'\backslash C)|]\geq
\frac{m|C'_2||C_1|}{8A_rn}+\frac{m|C_2||C'_1|}{8A_rn},\nonumber
\end{eqnarray}
which is at least $m|C|/16$ conditioned on the event $\mathcal{A}$
that $C'_1$ and $C'_2$ are both of size at least $A_rn/4$.
Therefore, by Hoeffdings inequality and the fact that
$\Pr[\neg\mathcal{A}]=O(n^{-3})$, we see that $|e(C,\bar{C})|\geq
|e(C,C'\backslash C)|\geq m|C|/32$ with probability at least
$1-e^{-m|C|/32}$.

On the other hand, $|C|=o(A_rn)$ with high probability. Therefore,
\begin{eqnarray}
\Pr[\exists R=o(r), \exists v, |e(C_{R}(v),\bar{C}_{R}(v))|\leq
m|C_R(v)|/32] \leq \sum_{k=1}^{o(A_rn)}{n\choose
k}e^{-mk/32}=o(1),\nonumber
\end{eqnarray}
where the last inequality follows from the assumption that $m\geq
K\ln n$, for some large constant $K$.

 Finally we note that $\vol(C_R(v))\leq
m|C_R(v)|+|e(C_R(v),\bar{C}_R(v))|$ and then we have
 \begin{eqnarray}
 \Phi(C_R(v))\geq
 \frac{m|C_R(v)|/32}{m|C_R(v)|+m|C_R(v)|/32}=\Omega(1),
 \end{eqnarray}
with high probability. This proves Theorem~\ref{thm:parameter}.
\end{proof}

Finally, we remark that the above proof can be adapted to the two generalized models $G_n^{\mathrm H}$ and $G_n^{\mathrm S}$. Since the number of long edges is relatively small compared with the number of long edges, the effect of long edges do not change the community structure too much. Specifically, to show that for $R=o(r)$, $C_R(v)$ is an expander in $G_n^{\mathrm H}$ and $G_n^{\mathrm S}$, we just need to use that $\vol(C_R(v))\leq
(m+1)|C_R(v)|+|e(C_R(v),\bar{C}_R(v))|$, and $|e(C_R(v),\bar{C}_R(v))|\geq m|C|/32$, which follows exactly the same as above.

\section{Conclusion}\label{sec:conclu}
We investigate the small-community phenomenon in networks and give
two models that unify the three typical properties of large-scale
networks: the power law degree distribution, the small-community
phenomenon and the small diameter property. The proposed network
models provide us insights of how real networks evolve and may have
potential applications in, e.g., wireless ad-hoc model and
sensor networks.

We have shown that the choice of parameters is subtle if one wants
all the three properties to coexist. The fundamental conflicts is
discussed, i.e., the power law degree distribution generated by the
preferential attachment scheme and the small diameter always lead to
an expander like graph, while the small-community phenomenon corresponds naturally to
anti-expander in some sense, which means that the conductance of
many subsets of small size is of order $o(1)$. Other reasons for
such conflicts worth further investigation.

Finally, our proof technique for the power law degree distribution
is of its own interest and it partially solves the open problems
in~\cite{FFV07:gpamodel1}~et al. It is interesting to find other
applications of this method, in particular, in the analysis of
randomized algorithms and network modeling.

\section*{Acknowledgments:}The research is partially supported by NSFC
distinguished young investigator award number 60325206, and its
matching fund from the Hundred-Talent Program of the Chinese Academy
of Sciences. Both authors are partially supported by the Grand
Project ``Network Algorithms and Digital Information" of the
Institute of Software, Chinese Academy of Sciences.

\medskip

\bibliographystyle{alphabetic}
\bibliography{geometric}

\newcommand{\etalchar}[1]{$^{#1}$}
\makeatletter
\@ifundefined{beginbibabs}{\def\beginbibabs{\begin{quotation}\noindent}
\def\endbibabs{\end{quotation}}}{}
\makeatother

\begin{thebibliography}{BKM{\etalchar{+}}00}

\bibitem[AJB99]{AJB99:diameterWWW}
R\'{e}ka Albert, Hawoong Jeong, and Albert-L\'{a}szl\'{o} Barab\'{a}si.
\newblock The diameter of the world wide web.
\newblock {\em Nature}, 401:130--131, 1999.

\bibitem[All04]{All04:alacrity}
Christopher Allen.
\newblock Life with alacrity: The dunbar number as a limit to group sizes.
\newblock 2004.

\bibitem[Avi08]{Avi08:rdg}
Chen Avin.
\newblock Distance graphs: from random geometric graphs to bernoulli graphs and
  between.
\newblock In {\em Proceedings of the fifth international workshop on
  Foundations of mobile computing}, DIALM-POMC '08, pages 71--78, New York, NY,
  USA, 2008. ACM.

\bibitem[Bol03]{Bo03:scalefree}
B.~Bollob{\'a}s.
\newblock Mathematical results on scale-free random graphs.
\newblock In {\em In Handbook of Graphs and Networks}, pages 1--34. Wiley-VCH,
  2003.

\bibitem[Bol01]{Bol01:randomgraphs}
B\'{e}la Bollob\'{a}s.
\newblock {\em Random Graphs, second edition}.
\newblock Cambridge University Press, 2001.

\bibitem[BKM{\etalchar{+}}00]{BKMRRSTW00:structure}
A.~Broder, R.~Kumar, F.~Maghoul, P.~Raghavan, S.~Rajagopalan, R.~Stata,
  A.~Tomkins, and J.~Wiener.
\newblock {Graph structure in the web}.
\newblock {\em Computer Networks}, 33(1-6):309--320, 2000.

\bibitem[CL04]{CL04:hybrid}
Fan Chung and Linyuan Lu.
\newblock The small world phenomenon in hybrid power law graphs.
\newblock In {\em Complex Networks, (Eds. E. Ben-Naim et. al.),
  Springer-Verlag}, pages 91--106. Springer, 2004.

\bibitem[CL06]{CL06:complex}
Fan Chung and Linyuan Lu.
\newblock {\em Complex Graphs and Networks (Cbms Regional Conference Series in
  Mathematics)}.
\newblock American Mathematical Society, Boston, MA, USA, 2006.

\bibitem[DL95]{DL95:degree}
Luc Devroye and Jiang Lu.
\newblock The strong convergence of maximal degrees in uniform random recursive
  trees and dags.
\newblock {\em Random Struct. Algorithms}, 7:1--14, August 1995.

\bibitem[DP09]{DP09:concentration}
Devdatt Dubhashi and Alessandro Panconesi.
\newblock {\em Concentration of Measure for the Analysis of Randomized
  Algorithms}.
\newblock Cambridge University Press, New York, NY, USA, 1st edition, 2009.

\bibitem[FFV07a]{FFV07:gpamodel1}
Abraham~D. Flaxman, Alan Frieze, and Juan Vera.
\newblock A geometric preferential attachment model of networks.
\newblock {\em Internet Mathematics}, 3(2), 2007.

\bibitem[FFV07b]{FFV07:gpamodel2}
Abraham~D. Flaxman, Alan~M. Frieze, and Juan Vera.
\newblock A geometric preferential attachment model of networks {II}.
\newblock {\em Internet Mathematics}, 4(1):87--111, 2007.

\bibitem[FG09]{FG09:powerlaw}
Pierre Fraigniaud and George Giakkoupis.
\newblock The effect of power-law degrees on the navigability of small worlds:
  [extended abstract].
\newblock In {\em Proceedings of the 28th ACM symposium on Principles of
  distributed computing}, PODC '09, pages 240--249, New York, NY, USA, 2009.
  ACM.

\bibitem[GR09]{GR09:small}
Georg Groh and Verena Rappel.
\newblock Towards demarcation and modeling of small sub-communities/groups in
  p2p social networks.
\newblock In {\em CSE '09: Proceedings of the 2009 International Conference on
  Computational Science and Engineering}, pages 304--311, Washington, DC, USA,
  2009. IEEE Computer Society.

\bibitem[HW56]{HR56:para}
Donald Horton and R.~Richard Wohl.
\newblock Mass communication and para-social interaction: Observations on
  intimacy at a distance.
\newblock {\em Psychiatry}, 19 (3):215--229, 1956.

\bibitem[Jor10]{Jor10:gpa}
Jonathan Jordan.
\newblock Degree sequences of geometric preferential attachment graphs.
\newblock {\em Adv.Appl.Prob.(SGSA)}, 42:319--330, 2010.

\bibitem[KVV04]{KVV04:clustering}
Ravi Kannan, Santosh Vempala, and Adrian Vetta.
\newblock On clusterings: Good, bad and spectral.
\newblock {\em J. ACM}, 51(3):497--515, 2004.

\bibitem[Kle00]{Kle00:smallworld}
J.~Kleinberg.
\newblock The small-world phenomenon: an algorithmic perspective.
\newblock In {\em Proceedings of the 32nd ACM Symposium on the Theory of
  Computing}. 2000.

\bibitem[KKR{\etalchar{+}}99]{KKRRT99:web}
Jon~M. Kleinberg, Ravi Kumar, Prabhakar Raghavan, Sridhar Rajagopalan, and
  Andrew~S. Tomkins.
\newblock The {Web} as a graph: measurements, models and methods.
\newblock In {\em Proceedings of the 5th Annual International Computing and
  Combinatorics Conference (COCOON)}, volume 1627 of {\em Lecture Notes in
  Computer Science}, pages 1--18, Tokyo, Japan, 1999. Springer.

\bibitem[KB10]{KB10:geocommunity}
Mikl\'{o}s Kurucz and Andr\'{a}s~A. Bencz\'{u}r.
\newblock Geographically organized small communities and the hardness of
  clustering social networks.
\newblock {\em Annals of Information Systems - Data Mining for Social Network
  Analyis}, 2010.

\bibitem[Lan05]{Lan05:spectral}
K.~Lang.
\newblock Fixing two weaknesses of the spectral method.
\newblock In {\em NIPS '05: Advances in Neural Information Processing Systems},
  volume~18, 2005.

\bibitem[LLDM08]{LLDM08:community}
Jure Leskovec, Kevin~J. Lang, Anirban Dasgupta, and Michael~W. Mahoney.
\newblock Community structure in large networks: Natural cluster sizes and the
  absence of large well-defined clusters.
\newblock {\em CoRR}, abs/0810.1355, 2008.
\newblock informal publication.

\bibitem[LLM10]{LLM10:empirical}
Jure Leskovec, Kevin~J. Lang, and Michael Mahoney.
\newblock Empirical comparison of algorithms for network community detection.
\newblock In {\em Proceedings of the 19th international conference on World
  wide web}, WWW '10, pages 631--640, New York, NY, USA, 2010. ACM.

\bibitem[LP11]{LP11:smallcommunity}
Angsheng Li and Pan Peng.
\newblock Communities structures in classical network models.
\newblock {\em Internet Mathematics}, 7(2):81--106, 2011.

\bibitem[Pen03]{Pen03:rgg}
M.~Penrose.
\newblock {\em Random Geometric Graphs}.
\newblock Oxford University Press, Oxford, 2003.

\bibitem[Pit94]{Pit94:rrt}
B.~Pittel.
\newblock Note on the heights of random recursive trees and random $m$-ary
  search trees.
\newblock {\em Random Structures and Algorithms}, 5:337--347, 1994.

\bibitem[RB03]{RB03:hierarchy}
Erzs\'{e}bet Ravasz and Albert-Laszlo Barab\'{a}si.
\newblock Hierarchical organization in complex networks.
\newblock {\em Physical Review E}, 67:026112, 2003.

\bibitem[SM95]{SM95:recursivetree}
R.~Smythe and H.~Mahmoud.
\newblock A survey of recursive trees.
\newblock {\em Theory of Probability and Mathematical Statistics}, 51:1--27,
  1995.

\bibitem[vdE08]{Esk08:geometric}
H.~van~den Esker.
\newblock A geometric preferential attachment model with fitness.
\newblock 2008.
\newblock Available at http://arxiv.org/abs/0801.1612.

\end{thebibliography}

\medskip

\end{document}